\documentclass{llncs}

\usepackage{amsmath,amsfonts}
\usepackage{url}
\usepackage{algorithmicx}
\usepackage[ruled]{algorithm}
\usepackage{algpseudocode}
\usepackage{algpascal}
\usepackage{algc}
\usepackage{algorithm}

\newcommand{\cO}{\mathcal{O}}
\newcommand{\E}{\mathbf{E}}
\newcommand{\M}{\mathbf{M}}
\newcommand{\s}{\mathbf{S}}
\newcommand{\D}{\mathbf{D}}
\newcommand{\F}{\mathbb{F}}

\newcommand{\EE}{\mathrm{E}}
\newcommand{\BL}{\mathrm{BL}}
\newcommand{\TE}{\mathrm{TE}}

\newcommand{\keywords}[1]{{
    \list{}{\advance\topsep by -5ex \relax\small \leftmargin=1cm
      \labelwidth=0pt \listparindent=0pt \itemindent\listparindent
      \rightmargin\leftmargin}\item[\hskip\labelsep \bfseries
    Keywords:] {#1} \endlist}}

\institute{}
\author{ Seyed Gholamhossein Hosseini, Reza Rezaeian Farashahi\inst{1,2} } \institute{
	Department of Mathematical Sciences, Isfahan University of Technology,\\
	Isfahan, 84156-83111, Iran
	\\{\tt farashahi@cc.iut.ac.ir,\; g.hoseinimath@gmail.com}\\
	\and
	School of Mathematics,
	Institute for Research in Fundamental Sciences (IPM),
	P.O. Box 19395-5746, Tehran, Iran\\
}

\begin{document}

\title{Differential Addition on Twisted Edwards Curves}

\maketitle

\begin{abstract}
This paper presents new differential addition (i.e., the addition of two points with the known difference) and doubling formulas, as the core step in Montgomery scalar multiplication, for
twisted Edwards curves. The formulas are provided with cost of $5\M+4\s+1\D$, $3\M+7\s+1\D$ and $3\M+6\s+3\D$
when the given difference point is in affine form. Here, $\M, \s, \D$ denote the costs of a field
multiplication, a field squaring and a field multiplication by a constant, respectively. 
\end{abstract}


\section{Introduction}
An elliptic curve $E$ over a field $\mathbb{F} $ is given by the Weiersras\ss\, equation
\[ y^2+a_1xy+a_3y=x^3+a_2x^2+a_4x+a_6\enspace \]
where coefficients $a_1, a_2, a_3, a_4$ and $a_6$ are in $\mathbb{F} $. Elliptic curves are represented in other forms 
such as Legendre equation, cubic equations, quartic equations and intersection of two quadratic surfaces~\cite{Silv,WDC8}.
Koblitz \cite{K87} and Miler \cite{M86} independently proposed the use of elliptic curves over finite fields in cryptography. Since the introduction of elliptic curve cryptography (ECC) elliptic curves over finite fields have been studied intensively and in particular, many proposals have been made to speed up their group arithmetic. ECC is one of the attractive asymmetric key cryptosystems  with the main advantage of achieving smaller key sizes under the same security level compare to that of other existing asymmetric systems such as RSA. This makes ECC suitable for software and hardware implementation in constrained environments including RFID tags, mobiles, sensors, and smart cards. 

The scalar multiplication is the main important operation of ECC which is implemented based on the basic operations in finite fields. 
That is to compute $kP$ for a given point $P$ on elliptic curve $E$ defined over a finite field $\F_q$ and
a given integer $k$. The scalar multiplication is performed recursively by point addition
and point doubling operations. One of the key factor in implementation of these basic curve operations is to reduce the number of field operations. This is why different forms of elliptic curves with several coordinates systems have been studied to improve the efficiency and to speed up the point multiplication. The well known recent form is Edwards curves \cite{Edw} and their variants (see~\cite{BBLP,BerLan1,BLRF,HWCD8}) with great impact to ECC.  

Side channel attacks use the time or power differences between implementing point addition and point doubling to reveal information about the bits of the secret $k$. Montgomery \cite{Mo} introduced a technique for scalar multiplication of points for a special type of curves in large characteristic that is known as Montgomery ladder. In each step of the Montgomery scalar multiplication algorithm both the addition and the doubling are used which makes this method resistant against simple side-channel attacks. For Montgomery curves, the basic formulas in each step of the Montgomery ladder is differential addition and doubling expressed only by the $x$-coordinates of the points. For the fixed point $P$ on the curve, this method computes the $x$-coordinate of the point $kP$ recursively by computing the $x$-coordinates of the points $P+2Q$ and $2Q$ from the $x$-coordinates of the points $P+Q$, $Q$. To avoid the costly field inversion operation, the computations are performed where points are represented in projective coordinates and the cost of projective $x$-coordinate formulas for Montgomery curves is $6\M+4\s+1\D$. Here $\M$, $\s$ and $\D$ are the costs of a field multiplication, a field squaring and a field multiplication by a constant, respectively. The $x$-coordinate of the fixed base point $P$ can be represented in affine form, then the differential mixed addition and doubling formulas are computed using $5\M+4\s+1\D$. 
 
The Montgomery method is extended to other forms of elliptic curves, where the basic operation in each step of the ladder is differential addition and doubling expressed only by suitable $w$-coordinates of the points. That is to compute the $w$-coordinates of the addition and doubling from the $w$-coordinates of given points and their difference. 
The Montgomery-like formulas for Edwards curves are presented in \cite{CGF}. Gaudry and Lubicz \cite{GL9} presents a very efficient Montgomery-like formulas for Kummer line the cost of $4\M+6\s+3\D$, and $3\M+6\s+3\D$ if the base point is affine. Bernstein and Lange \cite{BLEFD} extends the Kummer-line formulas for incomplete Edwards curves with the same costs. 

From the literature, the mixed differential addition and doubling formulas with the cost of $3\M+6\s+3\D$ are only given for elliptic curves with 3 points of order 2. Notice, complete twisted Edwards are suitable for cryptographic applications because of their fast complete addition law. A complete twisted Edwards curve has two points of order 4 and one single point of order 2. The main contribution of this paper is to provide faster Montgomary-like formulas for complete twisted Edwards curves, which covers all elliptic curves over finite fields with a point of order 4 and a single point of order 2. This paper presents new differential addition and doubling formulas for twisted Edwards curves with cost of $5\M+4\s+1\D$, $3\M+7\s+1\D$ and $3\M+6\s+3\D$ when the given difference point is in affine form. 

The rest of the paper is organized as follows. In \S 2 we review twisted Edwards curves, and in \S 3 we briefly describe differential addition on elliptic curves. The proposed new differential addition and doubling formulas are provided in \S 4, also, in \S 4 we present some example and an algorithm for recovery $x$ and $y$ coordinates of $kP$ when given $Q-P,w(P)$ and $w(Q)$,
 and finally, \S 5 concludes the paper with a comparison between our work and other previously related work.

Throughout the paper, the letter $p$ always denotes an odd prime number and $q$ denotes a prime power of $p$. 
A field is denoted by $\F$ and a finite field of size $q$ is denoted by $\F_q$. 
Let $\chi$ denote the quadratic character in $\F_q$, where $p\ge 3$. Then, for any $q$ where $p\ge 3$, we have
$u=w^2$ for some $w \in \F^*_q$ if and only if $\chi(u) =1$.

\section{Twisted Edwards curve}
In 2007, Edwards introduced a new normal form for elliptic curves \cite{Edw}. An original {\it Edwards curve}, defined over a field 
$\F$ with characteristic $p\ne2$, by the equation
\begin{equation*}
\label{eq:Edw} \E_{\EE, c}: \quad X^2 +Y^2 = c^2(1+X^2Y^2),
\end{equation*}
with $c\in \F$ and $c^5\ne c$. Bernstein and Lange \cite{BerLan1} considered the use of Edwards curves over finite fields for elliptic curve cryptography. 
They extended the original curves to the family of so called {\it Edwards curves}
\begin{equation*}
\label{eq:Edw-BL} \E_{\BL,d}: \quad X^2 +Y^2 = 1+ d X^2 Y^2,
\end{equation*}
where $d\in \F$ with $d\ne 0,1$.  The family of Edwards curves over a finite field $\F_q$ with odd characteristic is equivalent (up to $\F_q$ isomorphism) to the family of all elliptic curves over $\F_q$ with a $\F_q$-rational point of order 4 \cite{BBLP}. In other words, $\E_{\BL,d}(\F_q)$, the group of $\F_q$-rational points of the Edwards curve $\E_{\BL,d}$, has a $\F_q$-rational point of order 4 and in the other way around, every elliptic curve $E$ over  $\F_q$ with a point of order 4 can be represented as an Edwards curve. In addition, $\E_{\BL,d}(\F_q)$  has a single point of order $2$ if and only if $\chi(d)=-1$, i.e., the group $\E_{\BL,d}(\F_q)$ has three points of order $2$ if and only if $\chi(d)=1$.

%

Edwards curves and their extensions have attracted
great interest in elliptic curve cryptography (see~\cite{BBLP,BerLan1,BLRF,HWCD8}). 
Bernstein et.al. proposed the family of so-called {\it twisted Edwards}, \cite{BBLP}, given by
\begin{equation*}
\label{eq:Edw-T} \E_{\TE,a,d}: \quad aX^2 + Y^2 = 1+ d X^2Y^2,
\end{equation*}
where $a,d$ are distinct nonzero elements of $\F_q$. The addition and doubling law for $\E_{\TE,a,d}$ are given by 
\begin{equation}
\label{eq:TE}
\begin{array}{c}
(x_1,y_1)~,~(x_2,y_2)  \mapsto \left(\dfrac{x_1y_2+x_2y_1}  {1+dx_1x_2y_1y_2},\, \dfrac{y_1y_2-ax_1x_2} {1-dx_1x_2y_1y_2}
\right), \vspace*{3pt}\\
(x_1,y_1)  \mapsto \left(\dfrac{2x_1y_1}  {1+dx_1^2y_1^2},\,\dfrac{y_1^2-ax_1^2}  {1-dx_1^2y_1^2}  \right).\\
\end{array}
\end{equation}
The identity point of the addition law is $(0,1)$ and the additive negation of a point  $(x,y)$ is $(-x,y)$. 
The point $(0,-1)$ is a point of order 2.  If $\chi(a)=1$ then the points $(\pm 1/\sqrt{a}, 0) $ are of order 4. 

The projective closure of the twisted Edwards curve $\E_{\TE,a,d}$ in $\mathbb{P}^2$ includes the projective points $(X:Y:Z)$ in $\mathbb{P}^2(\F_q)$ satisfying the curve equation   
\begin{equation*}
 aX^2Z^2 + Y^2Z^2 = Z^4+ d X^2Y^2,
\end{equation*}
with the points at infinity $\infty_1=(1:0:0)$ and $\infty_2=(0:1:0)$. These points are singular. In the nonsingular model of $\E_{\TE,a,d}$ the point $\infty_1$  splits into two distinct $\F_q$-rational points if $\chi(ad)=1$ and is removed if $\chi(ad)=-1$. 
Similarly, above the point $\infty_2$ there exists exactly two distinct points if $\chi(d)=1$ and no point if $\chi(d)=-1$. 
So, if $\chi(d)=\chi(ad)=-1$ then the set of $\F_q$-rational projective points of $\E_{\TE,a,d}$ is the set of $\F_q$-rational affine points
which form a group. To represent the points above the points at infinity, the projective closures of $\E_{\TE,a,d}$  in $\mathbb{P}^3$ or in $\mathbb{P}\times \mathbb{P}$ are considered \cite{HWCD8,BL11}. The twisted Edwards curve $\E_{\TE,a,d}$ over $\F_q$ is represented by the set of points $(X:Y:T:Z)$ in $\mathbb{P}^3(\F_q)$ satisfying the equations  
\begin{equation*}
\quad aX^2 + Y^2 = Z^2+ dT^2,\qquad XY=ZT. 
\end{equation*}
Here, the $\F_q$-rational points above $\infty_1$ are 
$(1 : 0: \pm \sqrt{a/d} : 0)$  if $\chi(ad)=1$, and the points above $\infty_2$ are $(0: \pm \sqrt{d} : 1 : 0 )$ if $\chi(d)=1$. Hisil et.al.~\cite{HWCD8} gave the addition laws for the projective closure of $\E_{\TE,a,d}$  embedded in $\mathbb{P}^3$ as follows. 
\begin{gather} \label{Ad:TE1}
(X_1 : Y_1 : T_1: Z_1)+(X_2 : Y_2 : T_2 : Z_2)  \hspace*{37pt} \nonumber\\
=\begin{array}{lcl} \label{Eq:TEH1}
(\; &(X_1Y_2 + Y_1X_2)(Z_1Z_2-dT_1T_2) &: (Y_1Y_2-aX_1X_2)(Z_1Z_2 + dT_1T_2)  \\ 
&: (Y_1Y_2-aX_1X_2)(X_1Y_2 + Y_1X_2)   &: (Z_1Z_2-dT_1T_2)(Z_1Z_2 + dT_1T_2)\; ).
\end{array}
\end{gather}
Here the identity point is $(0:1:0:1)$ and the additive negation of a point  $(X:Y:T:Z)$ is $(-X:Y:-T:Z)$. 
The point $(0:-1:0:1)$ is a point of order 2 and the points $(1 : 0 : \pm \sqrt{a/d} : 0)$ are the points of order 2 if $\chi(ad)=1$.
The points $(\pm 1/\sqrt{a}: 0:0:1)$ and $(0: \pm \sqrt{d} : 1 : 0)$ are of order 4 if $\chi(a)=1$ and $\chi(d)=1$, respectively.
Other points of order 4 are $( \alpha : \beta :  \alpha\beta: 1)$ where $\alpha^4=1/ad$ and $\beta^4=a/d$.

Notice, that the family of twisted Edwards curves is the extension of the family of Edwards curves. 
Clearly, every Edwards curve $\E_{\BL,d}$ is the twisted Edwards $\E_{\TE,1,d}$. Furthermore, a
twisted Edwards curve $\E_{\TE,a,d}$ is a twist of the Edwards curve $\E_{\BL,\frac{d}{a}}$. 
Therefore, the family of twisted Edwards includes Edwards curves and their twists. 

The addition law in twisted Edwards curve $\E_{\TE,a,d}$ is complete if $\chi(d)=\chi(ad)=-1$. In other words, the projective formulas \eqref{Eq:TEH1}
have no exceptional cases if $\chi(a)=1$ and $\chi(d)=-1$ \cite{BBLP,HWCD8}. Here, we show that the addition law in twisted Edwards curve $\E_{\TE,a,d}$ is also complete if $\chi(a)=\chi(ad)=-1$. 

\begin{theorem}\label{thm:TED}
	Let $a$, $d$ be elements of $\F_q$ such that $ad(a-d)\ne 0$. Let
	$\E_{\TE,a,d}$ be a twisted Edwards curve over $\F_q$. Then, $\E_{\TE,a,d}$ has a complete projective formulas 
	over $\F_q$ if $\chi(ad)=-1$. 
\end{theorem}   
\begin{proof}
If $\chi(d)=\chi(ad)=-1$, then the projective formulas \eqref{Eq:TEH1}
are complete formulas for $\E_{\TE,a,d}$ \cite{BBLP,HWCD8}. If $\chi(a)=\chi(ad)=-1$, then the twisted Edwards curve 
$\E_{\TE,a,d}$ is birationally equivalent to $\E_{\TE,d,a}$ via the map $(x,y) \rightarrow (x,1/y)$. In other words, the projective points of the projective closures of $\E_{\TE,a,d}$ and $\E_{\TE,d,a}$ in $\mathbb{P}^3(\F_q)$ are corresponded to each other via the map $(X:Y:T:Z) \rightarrow (T: Z: X: Y)$. 
From \eqref{Eq:TEH1} and using the exchange of variables,  we obtain the projective formulas for the curve $\E_{\TE,a,d}$ as follows. 
\begin{gather}
(X_1 : Y_1 : T_1: Z_1)+(X_2 : Y_2 : T_2 : Z_2)  \hspace*{37pt} \nonumber\\
=\begin{array}{lcl} \label{Eq:TEH2}
(\; &(Z_1Z_2-dT_1T_2)(T_1Z_2 + Z_1T_2) &: (Y_1Y_2-aX_1X_2)(Y_1Y_2 + aX_1X_2) \\ 
&: (T_1Z_2 + Z_1T_2)(Y_1Y_2-aX_1X_2)   &: (Z_1Z_2-dT_1T_2)(Y_1Y_2 + aX_1X_2) \; ).
\end{array}
\end{gather}
Therefore, the projective formulas \eqref{Eq:TEH2} are complete formulas for $\E_{\TE,a,d}$ over $\F_q$ where $\chi(a)=-1$ and $\chi(d)=1$ which concludes the proof.\qed
\end{proof}	

It is shown in \cite{BBLP}, that a twisted Edwards curve $\E_{\TE,a,d}$ over a field $\F$ is birationally equivalent to a {\it Montgomery curve}
~\cite{Mo} given by the equation 
\begin{equation*}
\label{eq:Mont} \E_{\M,A,B}: \quad BY^2=X^3+AX^2+X,
\end{equation*}
where $A,B \in \F$ with $A\ne \pm2$ and $B\ne0$. In more details 
a twisted Edwards curve $\E_{\TE,a,d}$ is birationally equivalent to the Montgomery curve $\E_{\M,A,B}$ 
by the map $\psi:\E_{\TE,a,d} \to \E_{\M, A,B}$
\begin{equation}\label{eq:TEMont}
\psi(x,y)=\Big(\frac{1+y}{1-y},\frac{1+y}{x(1-y)} \Big).
\end{equation}
where $A=2(a+d)/(a-d),\, B=4/(a-d)$. 
Also, the Montgomery curve $\E_{\M,A,B}$ is birationally equivalent to the twisted Edwards
curve $\E_{\TE,a,d}$ by the inverse map 
\begin{equation}\label{MonTE}
\psi^{-1}(X,Y)=\Big(\frac{X}{Y},\frac{X-1}{X+1} \Big),
\end{equation}
where $a=(A + 2)/B,\, d = (A - 2)/B$.


\section{Differential addition}
\label{sec-dadd}
The main computational core for elliptic curve cryptography is performing scalar multiplication in an efficient and secure way. 
The computation of $kP$, for a given point $P$ on elliptic curve $E$ defined over a finite field $\F_q$ and
a given integer $k$, is performed recursively by point addition (PA) and point doubling (PD) formulas. The time or power differences between implementing point addition (PA) and point doubling (PD) can reveal information about the bits of the secret $k$ which makes the system insecure  against side channel attacks. 

In Montgomery curves \cite{Mo}, the special formulas for addition and doubling is done with the $X$ and $Z$ coordinates of a point in projective form. In each step of Montgomery ladder both addition and doubling are performed, which makes this method resistant against simple side-channel attacks. Recovering the $Y$ coordinate of the output point is done in the last step from the $X$ and $Z$ coordinates. 
Algorithm~\ref{dadM} provides the differential $x$-coordinate formulas for Montgomery curves $\E_{\M,A,B}$ over $\F_{q}$  \cite{Mo}.

\alglanguage{pseudocode}
\begin{algorithm}
	\caption{\; Projective $x$-coordinate dADD for Montgomery curves}\label{dadM}
	\begin{algorithmic}[1]
		\Statex {\bf Input : $\E_{\M,A,B}/\F_q : BY^2=X^3+AX^2+X$} \Comment{The Montgomery curve $\E_{\M,A,B}$ }
		\Statex {\bf \hspace{40pt} $(X_i: Z_i)=x(P_i),\; i=0,1,2.$} \Comment{$x(P_0)=x(P_1-P_2)$}
		\Statex {\bf Output : $(X_i: Z_i)=x(P_i),\; i=3,4.$} \Comment{$x(P_3)=x(P_1+P_2),\; x(P_4)=x(2P_1)$}
		\Statex {}
		\Function{\; \rm{d}ADD}{$(X_0: Z_0),\; (X_1: Z_1),\; (X_2: Z_2)$}
		\State $X_3=Z_0~(X_1X_2-Z_1Z_2)^2  $
		\State $Z_3=X_0~(X_1Z_2-X_2Z_1)^2$
		\State $X_4=(X_1^2-Z_1^2)^2$
		\State $Z_4=4X_1Z_1(~(X_1+Z_1)^2+(A-2)X_1Z_1~)$
		\State \Return $((X_4: Z_4),\; (X_3: Z_3))$\Comment{The differential addition and doubling }
		\EndFunction
	\end{algorithmic}
\end{algorithm}

We note, that $\cO=(0:1:0)$ is the point at infinity on the Montgomery curve $\E_{\M,A,B}$ over $\F_{q}$ and $x(\cO)$ in $\mathbb{P}(\F_q)$ is represented by $(1:0)$. Also, $x((0,0))$ is given by $(0:1)$. We can easily check, that the projective $x$-coordinate differential addition formulas in Algorithm~\ref{dadM} work for all inputs except for the case where $x(P_0)$ equals $(1:0)$ or $(0:1)$, i.e., where the point $P_0$ equals $\cO$ or $(0,0)$. In other words, the Montgomery ladder works for all inputs if the base point is not a point at infinity or the point $(0,0)$. The Montgomery ladder is given by the  
Algorithm~\ref{MonLD}, that for any integer $k$ and any point $P$ (not equal $\cO$ and $(0,0)$) computes $x(kP)$ correctly. 
In particular, the ladder works properly even if the integer $k$ is bigger than the order of the base point $P$. 
Therefore, one can use random scalar $k$ as a countermeasure to protect against differential power analysis attack.

\begin{algorithm}
	\caption{\; The modified Montgomery scalar multiplication}\label{MonLD}
	\begin{algorithmic}[1]
		\Statex {\bf Input : $\E_{\M,A,B}/\F_q : BY^2=X^3+AX^2+X$} \Comment{The Montgomery curve $\E_{\M,A,B}$ }
		\Statex {\bf \hspace{32pt} $P =(x :y : z)\in\E_{\M,A,B}(\F_q)$} \Comment{$P\ne \cO=(0:1:0), P \ne (0:0:1)$}
		\Statex {\bf \hspace{32pt} $k=(k_{m-1},\cdots, k_1,k_0)$} \Comment{$0\le k \in \mathbb{Z}$}
		\Statex {\bf \hspace{33pt} $(X_0:Z_0):=(x:z),\, (X_1:Z_1):=(1:0),\, (X_2:Z_2):=(x:z).$}
		\Statex {\bf Output : $x(kP)$}
		\Statex {}
		\For{$i:=m-1$ {\bf down to} $0$}
		\If{$k_i=0$}
		\State {$((X_1:Z_1), (X_2:Z_2)):=\text{dADD}((X_0: Z_0),(X_1: Z_1),(X_2: Z_2))$}
		\Else
		\State {$((X_2:Z_2), (X_1:Z_1)):=\text{dADD}((X_0: Z_0),(X_2: Z_2),(X_1: Z_1))$}
		\EndIf
		\EndFor
		\State \Return $(X_1: Z_1),\; (X_2: Z_2)$ \Comment{The differential addition and doubling }
	\end{algorithmic}
\end{algorithm}

The Montgomery method is extended to other forms of elliptic curves with a suitable rational function.   
Let $w$ be a rational function in the coordinate ring of the elliptic curve $E$ over $\F_q$ where $w(P)=w(-P)$ for every point $P$ in $E(\F_q)$. 
The $w$-coordinate \emph{differential addition} and \emph{doubling} means to compute $w(P+Q)$ and $w(2Q)$ from
given values $w(P)$, $w(Q)$ and $w(P-Q)$, where $P, Q$ are points on $E(\F_q)$. 
If $w$ is regular at the point $P$ then $w(P)$ is represented by $(w(P):1)$ in the projective line ${\mathbb{P}}(\F_q)$. Otherwise, it is
represented by $(1:0)$. For the fixed point $P$ on the curve and a positive integer $k$, the $w$-coordinate of the point $kP$ is performed recursively by differential addition and doubling formulas expressed only by $w$-coordinates of the points. 
 
A projective $w$-coordinate differential addition is \emph{complete} if it works for all inputs. 
Also, it is \emph{almost complete} if the $w$-coordinate differential formulas work for all inputs 
except for the case where $w(P_0)$ equals $w(\cO)$, where $\cO$ is the neutral element of the group of points $E(\F_q)$. 
Note that, the projective $x$-coordinate differential addition for Montgomery curves given in Algorithm~\ref{dadM} 
works for all inputs except for the case where $w(P_0)$ equals $(1:0)$ or $(0:1)$. 
The fast and complete differential addition formulas are very interesting for implementations. 
But, if the base point $P_0$ has large prime order then with suitable $w$-function 
$w(P_0) \ne w(\cO)$ and $w(P_0) \ne (1:0), (0:1)$. Therefore, the almost complete and Montgomery-like formulas are usable for cryptographic applications.   

The cost of projective $x$-coordinate differential addition and doubling formulas for Montgomery curves $\E_{\M,A,B}$ over $\F_{q}$ given by Algorithm~\ref{dadM} is $6\M+4\s+1\D$, where a multiplication in $\F_q$ costs one $\M$, a squaring costs one $\s$ and the cost of field multiplication by a parameter (as a constant) is denoted by $\D$. The $x$-coordinate of the fixed base point $P$ can be represented by  $x(P)=(X_0 : Z_0)$, where $Z_0=1$, then the differential addition and doubling formulas are computed using $5\M+4\s+1\D$. 

Castryck, Galbraith and Farashahi \cite{CGF} give the $y$-coordinate differential addition Montgomery-like formulas for Edwards curves. They use the quasi free projective map between twisted Edwards and Montgomery curves which provides the Montgomary formulas for twisted Edwards curves with the cost of $6\M+4\s+1\D$, and $5\M+4\s+1\D$ if the base point is affine. They also give a doubling formulas with cost of $1\M+3\s+3\D$ assuming $d$ is a square element.  Gaudry and Lubicz \cite{GL9} obtained a very efficient differential addition Montgomery-like formulas for Kummer line with the cost of $4\M+6\s+3\D$, and $3\M+6\s+3\D$ if the base point is affine. The Kummer line behaves very similar to the Montgomery form. Compare to the Montgomery form, the Kummer line formulas saves $2\M-2\s$, but have extra $2$ multiplication by constants. The Kummer line is linked to the Legendre curve $\E_\lambda: Y^2=X(X-1)(X-\lambda)$, where $\lambda=a^4/(a^4-b^4)$ and $(a:b)$ defines the Kummer line. The group order of the corresponding curve $E_\lambda$ over $\F_q$ is divisible by 4, and in particular it has 3 points of order 2. Bernstein and Lange \cite{BLEFD} provides a Kummer-line formulas for Edwards curves $\E_{\BL,d}$ where $d=r^2$ is a square element. They give the cost of $w$-coordinates mixed differential addition and doubling formulas for $w=ry$ and $w=ry^2$ by $3\M+6\s+5\D$ and $3\M+6\s+3\D$ respectively. Here, the Edwards curve $\E_{\BL,d}$ over $\F_q$ with $\chi(d)=1$ has 3 points of order 2 and the addition law is not complete. In the next section, we provide new Montgomary-like formulas for complete twisted Edwards curves. 


\vspace*{5pt}
\noindent \textbf{$w$-coordinate dADD.} Throughout the paper, for the function $w$ on elliptic curve $E$ over $\F$, and for the points $P_1, P_2 \in E(\F)$, let  
$w_0=w(P_2-P_1)$, $w_1=w(P_1)$, $w_2=w(P_2)$, $w_3=w(P_1+P_2)$ and $w_4=w(2P_1)$. Also, for $i=0,1,2,3,4$, $w_i$ are represented by $(W_i : Z_i)$ in the projective line $\mathbb{P}(\F).$
\section{Differential addition on twisted Edwards curve}\label{sec diff}
In this section, we provide new differential addition and doubling formulas for twisted Edwards. The  mixed formulas have the cost $5\M+4\s+1\D$, $3\M+7\s+1\D$. In addition, we give mixed formulas with cost of $3\M+6\s+3\D$ for subfamily of twisted Edwards curves. 
These efficient and fast formulas are applicable for complete twisted Edwards in this subfamily. From the birational map between the twisted Edwards and Montgomery curve, we can use similar formulas for Montgomery curves.   

\subsection{Twisted Edwards}
Here, we consider \textit{twisted Edwards curves} $\E_{\TE,a,d}$  and present new
\textit{$w$}-coordinates differential formulas.

\begin{proposition}\label{P:TEDW1}
	Let $w$ be the function on twisted Edwards curve  $\E_{\TE,a,d}$ over $\F_q$ given by $w(x,y)=dx^2y^2$. Let $P_1$ , $P_2$ be two points of twisted Edwards curve  $\E_{\TE,a,d}$. Consider the $w$-coordinate dADD for $P_1$ , $P_2$. If  $w_1w_2 \ne 1 $ and $w_1^2 \ne 1$, then we have 
	\begin{equation}\label{eq:A1}
	w_4	= \dfrac{4w_1((w_1+1)^2-ew_1)}{(w_1^2-1)^2 }~,
	\end{equation}
	
	\begin{equation}\label{eq:BB1}
	w_3w_0 = \dfrac{(w_1-w_2)^2}{(w_1w_2-1)^2}~,
	\end{equation}
and
\begin{equation}\label{eq:C1}
w_3+w_0 = \dfrac{2(w_1+w_2)(w_1w_2+1)-4(e-2)w_1w_2}{(w_1w_2-1)^2},
\end{equation}
where $e=4a/d$.
\end{proposition}

\begin{proof}
	This $w$-function is well computed for all affine points on a twisted Edwards curve .
	Since $-(x,y)=(-x,y)$, for all points $P$ on the curve, we have
	$w(P)=w(-P)$. Also, we have $w(\cO)=0$.
	Let $P_1=(x_1,y_1)$ be a point of twisted Edwards curve  $\E_{\TE,a,d}$. Then from the doubling formula \eqref{eq:TE} we have
	\begin{equation*}
	w_4=d(x_4y_4)^2=\dfrac{4dx_1^2y_1^2(ax_1^2-y_1^2)^2}{(1+dx_1^2y_1^2)^2(1-dx_1^2y_1^2)^2}.
	\end{equation*}
	
	We can write $(ax_1^2-y_1^2)^2=(ax_1^2+y_1^2)^2-4ax_1^2y_1^2=(1+dx_1^2y_1^2)^2-4ax_1^2y_1^2$.
	So by substitution $w_1=dx_1^2y_1^2$ and $e=4a/d$, we have $(ax_1^2-y_1^2)^2=(1+w_1)^2-ew_1$. Then we have
	\begin{equation*}
	w_4	= \dfrac{4w_1((w_1+1)^2-ew_1)}{(w_1^2-1)^2 }.
	\end{equation*}
	Also from the addition formula \eqref{eq:TE}, we have 
	\begin{equation*} \label{eq:PAdd1}
	w_3w_0=\dfrac{\left(d x_1^2y_1^2(y_2^4+a^2x_2^4)-dx_2^2y_2^2(y_1^4+a^2x_1^4) \right)^2} {\left(1-(dx_1^2y_1^2)(dx_2^2y_2^2) \right)^4}~~,
\end{equation*}
And
\begin{equation*}
w_3+w_0=\dfrac{2d x_1^2y_1^2(y_2^4+a^2x_2^4)+2dx_2^2y_2^2(y_1^4+a^2x_1^4)-8adx_1^2y_1^2x_2^2y_2^2} {\left(1-d^2x_1^2y_1^2x_2^2y_2^2 \right)^2}.
	\end{equation*}
	
	Now by substitution $y_i^4+a^2x_i^4=(1+dx_i^2y_i^2)^2-2ax_i^2y_i^2$ and $w_i=dx_i^2y_i^2$  for $i=1,2$ in above equations, we obtain 
	$$ w_3w_0=\dfrac{\left(w_1((1+w_2)^2-2a/d~w_2)-w_2((1+w_1)^2-2a/d~w_1)\right)^2}  {(1-w_1w_2)^4},$$
	$$ w_3+w_0=\dfrac{2w_1((1+w_2)^2-2a/d~w_2)+2w_2((1+w_1)^2-2a/d~w_1)-8a/dw_1w_2}  {(1-w_1w_2)^2}.$$
	Now by some straightforward calculation, we obtain 
	$$w_3w_0=\dfrac{\left( w_1w_2(w_1-w_2) -(w_1-w_2)\right)^2}  {(1-w_1w_2)^4}= \dfrac{(w_1-w_2)^2}{(w_1w_2-1)^2},$$
	$$ w_3+w_0=\dfrac{2(w_1+w_2)(w_1w_2+1)-4(e-2)w_1w_2}{(w_1w_2-1)^2}.$$ \qed
	%
	%
\end{proof}
Assume that $w_0$ is given as a field element, and the inputs $w_1, w_2$ are given as fractions
$W_1/Z_1$ , $W_2/Z_2$ and the outputs $w_4, w_3 $ are given as fraction $W_4/Z_4$ and  $W_3/Z_3$. From
Eq.~\eqref{eq:A1} and Eq.~\eqref{eq:BB1} the explicit projective formulas are given by
\begin{equation}\label{eq:P1}
\begin{array}{c}
 \dfrac{W_4}{Z_4}= \dfrac{4W_1Z_1(~(W_1+Z_1)^2-eW_1Z_1)} {(W_1-Z_1)^2(W_1+Z_1)^2}, \vspace*{7pt}\\
 \dfrac{W_3}{Z_3}= \dfrac{Z_0~(W_1Z_2-W_2Z_1)^2}  {W_0~(W_1W_2 -Z_1Z_2)^2}\enspace. 
\end{array}
\end{equation}
The cost of this projective $w$-coordinates addition and doubling formulas is $6\M+4\s+1\D$. If we set $Z_0=1$, then the mixed projective
$w$-coordinates differential addition and doubling formulas have the total cost $5\M+4\s+1\D $ as
follows:
\begin{equation}
\label{eq:P2}
\begin{array}{c}
A_1=(W_1+Z_1) ,\ B_1=(W_1-Z_1),\ A_2=(W_2+Z_2),\ B_2=(W_2-Z_2), \vspace*{3pt}\\
C=A_1B_2~ ,~\ D=A_2B_1~,~\ 	E=A_1^2-B_1^2 , \vspace*{3pt}\\
W_4	= E(A_1^2-(e/4)~E)~,~Z_4=A_1^2B_1^2, \vspace*{3pt}\\
W_3=(C-D)^2~,~Z_3=w_0(C+D)^2 
\enspace. 
\end{array}
\end{equation}
From \eqref{eq:P2}, the costs of differential addition and doubling formulas are $3\M+2\s $ and $2\M+2\s+1\D $, respectively. 
And, the total cost of the mixed differential addition and doubling is 
$5\M+4\s+1\D$. In addition, the cost of following mixed differential addition and doubling formulas is $3\M+7\s+1\D$, 
where $1\D$ is the multiplication by the parameter $e/4=a/d$. So, if the parameter $e$ is chosen to be small then the cost of mixed differential formulas is $3\M+7\s$. 
\begin{equation}
\label{eq:P4}
\begin{array}{c}
 A_1=(W_1+Z_1) ,\ B_1=(W_1-Z_1),\ A_2=(W_2+Z_2),\ B_2=(W_2-Z_2), \vspace*{3pt}\\
 C=A_1B_2~ ,~\ D=A_2B_1~,~\ 	E=A_1^2-B_1^2 ~,~F=(A_1^4+B_1^4)-E^2, \vspace*{3pt}\\
W_4	= 2(A_1^4-(e/4)E^2)-F~~~,~~~Z_4=F, \vspace*{3pt}  \\
 W_3=(C-D)^2~~~,~~~Z_3=w_0(C+D)^2 \enspace.  
\end{array}
\end{equation}
Furthermore, for the twisted Edwards curves $\E_{\TE,a,d}$ with $\chi(e(e-4))=\chi(a(a-d))=1$,  
the cost of the following mixed differential addition and doubling formulas is  $3\M+6\s+3\D$. Here we let $r^2=(e-4)/e$.  
\begin{equation}
\label{eq:P6}
\begin{array}{c}
A_1=(W_1 + Z_1)~,~ B_1=(W_1 - Z_1)~,~A_2=(W_2 + Z_2)~, B_2=(W_2 -Z_2), \vspace*{3pt}\\
C= A_1~B_2~,~D= A_2~B_1~,~H_1=(rA_1^2+B_1^2)^2~,~ H_2=(rA_1^2-B_1^2)^2, \vspace*{3pt}\\
G=(H_1+H_2)\ ,\ K=(H_1-H_2)\ ,\ S=\frac{1}{r}K\ ,\ T=rK, \vspace*{3pt}\\
W_4=2G-S-T~~,~~Z_4=T-S, \vspace*{3pt}\\
W_3=(C-D)^2~~,~~Z_3=w_0(C+D)^2 \enspace.
\end{array}
\end{equation}
From differential addition and doubling formulas \eqref{eq:P6}, the costs of differential
addition and doubling are $3\M+2\s $, $4\s+3\D $ respectively. And, the total cost of the mixed differential addition and doubling formulas is
$3\M+6\s+3\D$, where $2\D$ is the multiplication by the parameter $r$ and one $\D$ is the multiplication by $1/r$. So, if the parameter $r$ is chosen to be small then the cost of mixed differential formulas is $3\M+6\s+1\D$. 
\begin{proposition}\label{P:TEDW01}
Fix a field  $\F_q$ with odd charactristic and $a,d \in \F_q $ such that $ad(a-d) \ne 0$.
Let $\infty_1=(0:\pm \sqrt{d}:1:0)$ and  $\infty_2=(1:0:\pm \sqrt{a/d}:0)$.Define $w:\E_{\TE,a,d}(\F_q) \mapsto \F_q  \cup \{\infty \} $ as fallow : $w(x,y)=dx^2y^2; ~w(\infty_1)=w(\infty_2)=\infty$. 

Let $W_i,Z_i~(i=0,...,4)$  be elements of $\F_q$ and $e=4a/d$. Define
\begin{equation*}
\label{eq:PJ0}
\begin{array}{c}
W_4= 4W_1Z_1(~(W_1+Z_1)^2-eW_1Z_1), \vspace*{7pt}\\
Z_4= (W_1^2-Z_1^2)^2,~~~~~~~~~~~~~~~~~~~~~~\vspace*{7pt}\\
W_3= Z_0~(W_1Z_2-W_2Z_1)^2,~~~~~~~~~~~~\vspace*{7pt}\\
Z_3= W_0~(W_1W_2 -Z_1Z_2)^2.~~~~~~~~~~~\enspace 
\end{array}
\end{equation*}
a) Let $P_1$ be an element of $\E_{\TE,a,d}$. If $W_1 \ne 0 $ or $ Z_1 \ne 0 ;~w(P_1)=W_1/Z_1$, then $W_4/Z_4 \ne 0/0$ and $w(2P_1)=W_4/Z_4$. \\
b) Assume that $W_0 Z_0\ne 0;~w(P_0)=W_0/Z_0. ~W_1 \ne 0 $ or $ Z_1 \ne 0 ;~w(P_1)=W_1/Z_1 .$ 
 $W_2 \ne 0$ or $ ~Z_2 \ne 0;~w(P_2)=W_2/Z_2$,  then $W_3/Z_3 \ne 0/0$ and $w(P_1+P_2)=W_3/Z_3$. 

Here $W/Z$ means the quotient of $W$ and $Z$ in $\F_q$  if $Z \ne 0$; it means $\infty$ if $W \ne 0$ and $Z=0$; it is undefined if $W=Z=0$.
\end{proposition}
\begin{proof}
a)
If $Z_1=0$ then $Z_4=W_1^4\ne 0 $ and $W_4=0$ so $(W_4:Z_4) = (0:1)$ and  $W_4/Z_4 \ne 0/0$
as claimed.
Also  $Z_1=0$ so $P_1= \infty_1$ or $ \infty_2$ and $2P_1=\cO$ hence $w(2P_1)=(0:1)$ as claimed.
If $W_1=0$ then $Z_4=Z_1^4\ne 0 $ and $W_4=0$ so $(W_4:Z_4) = (0:1)$ and  $W_4/Z_4 \ne 0/0$
as claimed.
from  $W_1=0$ we have $x_1=0$ or $y_1=0$. So $P_1=(0,\pm1)$ or $P_1=(\pm 1/\sqrt{a},0)$. So $2P_1= \cO $ or $2P_1=(0,-1)$ so $w(2P_1)=(0:1)$ as claimed.

Assume from now on that $W_1~Z_1 \ne 0$; then $w(P_1) \ne \infty$. 
If $W_4/Z_4=0/0$, then we have $(W_1+Z_1)^2-eW_1Z_1=0$ and $W_1^2-Z_1^2=0$. So $W_1=\pm Z_1$. 
 Now by substitution $W_1=-Z_1$ in equation $(W_1+Z_1)^2-eW_1Z_1=0$ we have $W_1Z_1=0$, contradicting the hypothesis that $W_1Z_1 \ne 0$. on the other hands, if $W_1=Z_1$, then we have $(a-d)W_1=0$, contradicting the hypothesis that $W_1\ne 0$ and $ad(a-d) \ne 0$. So  $W_4/Z_4 \ne 0/0$ as claimed. Also from  $W_1~Z_1 \ne 0$ we  know that, the point of $P_1$ is affine point, so from proposition \ref{P:TEDW1}  we have  $w(2P_1)=W_4/Z_4$. \vspace*{7pt}\\
 
b)
If $P_2=P_1$ then $P_0=\cO$ so $W_0=0$, contradiction. Hence $P_1 \ne P_2$.

If $P_2=-P_1$ then $W_1/Z_1=w(P_1)=w(-P_2)=w(P_2)=W_2/Z_2$ so $W_1Z_2=W_2Z_1$ so $W_3=0$. We will show that $Z_3 \ne 0$, hence $(W_3:Z_3) = (0:1)$ and  $W_3/Z_3 \ne 0/0$ as claimed.
Note that $W_1 \ne 0:$ if  $W_1 = 0$ then  $Z_1 \ne 0$ so  $w(P_1) =0$ so $P_1=\cO$ or $(0,-1)$ so $P_2=-P_1=P_1$, contradiction. similary $W_2 \ne 0$.
Now  suppose that  $Z_3=0$. Then $W_0(W_1W_2-Z_1Z_2)=0$, but $W_0 \ne 0$, so $(W_1W_2-Z_1Z_2)=0$ so $\frac{W_2}{Z_2}=\frac{Z_1}{W_1}$.  In addition, we have $\frac{W_2}{Z_2}=\frac{W_1}{Z_1}$. Hence $\frac{W_1}{Z_1}=\frac{Z_1}{W_1}$, so $W_1^2-Z_1^2=0$. So from doubling formula we have $Z_4=0$, But from $P_2=-P_1$ we hve  $P_0=2P_1$, so  $Z_0=Z_4=0$, contradiction. Also we have $P_1+P_2=\cO$, so $w(P_3)=w(\cO)=(0:1)$ as calimed.

If $P_1= \infty_1,~P_2= \infty_2$ (or $P_1= \infty_2,~P_2= \infty_1$ ) then $Z_1=Z_2=0$ so $Z_3=W_0(W_1W_2)^2 \ne 0$ and $W_3=0$, so $(W_3:Z_3) = (0:1)$, hence  $W_3/Z_3 \ne 0/0$ as claimed. Also from addition formula \eqref{Ad:TE1} 
$w(P_1+P_2)=(0:1)$.

If $P_1$ be a point at infinity and $P_2$ be an affine point,  then $Z_1=0$ , $Z_2 \ne 0$. We know that $w(P_1)=(1:0)$, 
so $W_3=Z_0Z_2^2$ and $Z_3=W_0W_2^2$. If $W_3=0$ then we have $Z_2=0$, contradiction. Hence   $W_3/Z_3 \ne 0/0$.

Assume from now on that  $P_1,P_2 \ne \infty_1,~\infty_2$ and $P_2 \ne \pm P_1$. If $Z_3 = 0$, then $W_0(W_1W_2-Z_1Z_2)=0$, but $W_0Z_0 \ne 0$, so  $W_1W_2-Z_1Z_2=0$. We know that, $P_1$ and $P_2$ are affine points, so $Z_1 Z_2\ne 0$. So $(W_1/Z_1)(W_2/Z_2)=1$, hence $w(P_1)w(P_2)=1$. But from proposition \ref{P:TEDW1} we have $w_1w_2-1 \ne 0$ , so $Z_3 \ne 0$ and $W_3/Z_3 \ne 0/0$ as claimed. On the other hand, the points of $P_1$ and $P_2$ are affine points, so from proposition \ref{P:TEDW1}  we have  $w(P_1+P_2)=W_3/Z_3$. ~~~ \qed
\end{proof}
\begin{proposition}\label{P:TEDW001}
 Let $a,d$ be elements of $\F_q$ such that $ad(a-d) \ne 0$ and $\chi(d)=\chi(ad)=-1$.  
	Let $w$ be the function on twisted Edwards curve    over $\F_q$ given by $w(x,y)=dx^2y^2$. The following differential addition formula is complete.
\begin{gather}\label{eq:D01}
w_3+w_0 = \dfrac{2(w_1+w_2)(w_1w_2+1)-4(e-2)w_1w_2}{(w_1w_2-1)^2}~.
\end{gather}
	where $e=4a/d$.
\end{proposition}

\begin{proof}
 If $P_1=P_2$ then  $w_1=w_2$ and $w_0=0$. So $w_3=w_4.  $
We need to show that the denominators are nonzero. 
Suppose that $w_1w_2-1=0$. Then $d^2x_1^2y_1^2x_2^2y_2^2=1$ and $x_iy_i \ne 0$. So $dx_1y_1x_2y_2=\pm 1$. 
If $x_1y_1=1/(dx_2y_2)$, then substituating into curve equation yields
$$ax_1^2+y_1^2=\dfrac{ax_2^2+y_2^2}{dx_2^2y_2^2}.$$
 Therefore,
 $$(\sqrt{a}x_1+y_1)^2=\dfrac{ax_2^2+y_2^2}{dx_2^2y_2^2}+2\sqrt{a}x_1y_1=\frac{1}{d}\dfrac{(\sqrt{a} x_2+y_2)^2}{(x_2y_2)^2}.$$
 Since $\chi(d)=-1$, this must reduce to $0=0$. So $\sqrt{a}x_1+y_1=0.$ 
 
 Similary, 
 $$(\sqrt{a}x_1-y_1)^2=\frac{1}{d}\dfrac{(\sqrt{a} x_2-y_2)^2}{(x_2y_2)^2}.$$
 Which implies that $\sqrt{a}x_1-y_1=0.$ Therefore, $x_1=y_1=0$. which is a contradiction.
 
 The case where $dx_1y_1x_2y_2=-1$ similary produces a contradiction. Therefore the differential addition and doubling formulas is always defined for points in $\E_{\TE,a,d}(F_q).~$ \qed
 \end{proof}

From equation \eqref{eq:D01} we obtain the following  explicit projective formula given by
\begin{equation*}\label{eq:PC1}
\begin{array}{c}
 \dfrac{W_3}{Z_3}+\dfrac{W_0}{Z_0}= \dfrac{ 2(W_1Z_2+W_2Z_1)(W_1W_2+Z_1Z_2)-4(e-2)W_1W_2Z_1Z_2 }  {(W_1W_2 -Z_1Z_2)^2}\enspace. 
\end{array}
\end{equation*}
In the next proposition, we show that if $\chi(d)=\chi(ad)=-1$, then this differential addition formula is complete.

If we set $Z_0=1$, then  the mixed projective
$w$-coordinates differential addition and doubling formulas have the total cost of $7\M+3\s+2\D $ as
follows:

\begin{equation}
\label{eq:PC02}
\begin{array}{c}
 A_1=(W_1+Z_1) ,\ B_1=(W_1-Z_1), \ E_1=A_1^2-B_1^2 , \vspace*{3pt}\\

A_2=W_1W_2 ,\ B_2=Z_1Z_2,\ E_2=(W_1+Z_1)(W_2+Z_2), \vspace*{3pt}\\
C=A_2+B_2,\ D=A_2-B_2, \vspace*{3pt}\\
W_4=E_1(A_1^2-(e/4)~E_1)~,~Z_4=A_1^2B_1^2, \vspace*{3pt}\\

W_3=C(2E_2-e\ C)-(w_0-e+2)D^2~,~Z_3=D^2 
\enspace. 
\end{array}
\end{equation}

Also,  the following  mixed projective
$w$-coordinates differential addition and doubling formulas have the total cost of $5\M+6\s+2\D $ :

\begin{equation}
\label{eq:PC2}
\begin{array}{c}
A_1=(W_1+Z_1) ,\ B_1=(W_1-Z_1), \ E_1=A_1^2-B_1^2 ,\vspace*{3pt}\\
A_2=W_1W_2 ,\ B_2=Z_1Z_2,\ E_2=(W_1+Z_1)(W_2+Z_2), \vspace*{3pt}\\
C=A_2+B_2,\ D=A_2-B_2, \vspace*{3pt}\\
W_4	= A_1^4-B_1^4+(1-e/2)E_1^2~,~Z_4= A_1^4+B_1^4-E_1^2, \vspace*{3pt}  \\
W_3=C(2E_2-e\ C)-(w_0-e+2)D^2~,~Z_3=D^2  \enspace. 
\end{array}
\end{equation}
\begin{proposition}\label{P:TEDW0001}
Fix a field  $\F_q$ with odd charactristic and $a,d \in \F_q $ such that $ad(a-d) \ne 0$ and $\chi(d)=\chi(ad)=-1$.
Define $w:\E_{\TE,a,d}(\F_q) \mapsto \F_q \cup  \{\infty \}$ as fallow : $w(x,y)=dx^2y^2$. 

Let $W_i,Z_i~(i=0,...,4)$  be elements of $\F_q$ and $e=4a/d$. Define
\begin{equation*}
\label{eq:PJ001}
\begin{array}{c}
W_4= 4W_1Z_1(~(W_1+Z_1)^2-eW_1Z_1), \vspace*{7pt}\\
Z_4= (W_1^2-Z_1^2)^2,~~~~~~~~~~~~~~~~~~~~~~\vspace*{7pt}\\
W_3= Z_0~(2(W_1Z_2+W_2Z_1)(W_1W_2+Z_1Z_2)-4(e-2)W_1W_2Z_1Z_2)-W_0(W_1W_2 -Z_1Z_2)^2,~\vspace*{7pt}\\
Z_3= Z_0~(W_1W_2 -Z_1Z_2)^2 .~\enspace 
\end{array}
\end{equation*}
a) Let $P_1$ be an element of $\E_{\TE,a,d}$ and  $w(P_1)=W_1/Z_1$, then $W_4/Z_4 \ne 0/0$ and $w(2P_1)=W_4/Z_4$. \\
b) Let $P_1$ and $P_2$ be two element of $\E_{\TE,a,d}$ and  $w(P_1)=W_1/Z_1 ,w (P_2)=W_2/Z_2$,  then $W_3/Z_3 \ne 0/0$ and $w(P_1+P_2)=W_3/Z_3$. 

Here $W/Z$ means the quotient of $W$ and $Z$ in $\F_q$  if $Z \ne 0$; it means $\infty$ if $W \ne 0$ and $Z=0$; it is undefined if $W=Z=0$.
\end{proposition}
\begin{proof}
Since $\chi(d)=\chi(ad)=-1$, the set of $\F_q$-rational projective points of $\E_{\TE,a,d}$ is the set of $\F_q$-rational affine points which form a group. So for every point $P$, we have $Z \ne 0$. 

a) If $W_1~Z_1 = 0$, then $W_1=0$, so $Z_4=Z_1^4 \ne 0 $ and $W_4=0$. So $(W_4:Z_4) = (0:1)$ and  $W_4/Z_4 \ne 0/0$ as claimed.
From  $W_1=0$ we have $x_1=0$ or $y_1=0$. So $P_1=(0,\pm1)$ or $P_1=(\pm 1/\sqrt{a},0)$. So $2P_1= \cO $ or $2P_1=(0,-1)$ so $w(2P_1)=(0:1)$ as claimed.

If $W_1~Z_1 \ne 0$ and $W_4/Z_4=0/0$, then we have $(W_1+Z_1)^2-eW_1Z_1=0$ and $W_1^2-Z_1^2=0$.
From $W_1^2-Z_1^2=0$ we have $W_1=\pm Z_1$.
 Now by substitution $W_1=-Z_1$ in equation $(W_1+Z_1)^2-eW_1Z_1=0$ we have $W_1Z_1=0$, contradicting the hypothesis that $W_1Z_1 \ne 0$. On the other hands, if $W_1=Z_1$, then we have $(a-d)W_1=0$, contradicting the hypothesis that $W_1\ne 0$ and $ad(a-d) \ne 0$. So  $W_4/Z_4 \ne 0/0$ as claimed. We  know that, the point of $P_1$ is affine point, so from proposition \ref{P:TEDW1}  we have  $w(2P_1)=W_4/Z_4$ as claimed.

b) Assume that $Z_3 =0 $, so  $Z_0(W_1W_2-Z_1Z_2)=0$. So $(W_1W_2-Z_1Z_2)=0$. We know that, $P_1$ and $P_2$ are affine points, so $Z_1 Z_2\ne 0$, hence $W_1W_2 \ne 0$. Then $(W_1/Z_1)~(W_2/Z_2)-1=0$. So $w(P_1)w(P_2)-1=0$, where from proposition \ref{P:TEDW001} is a contradiction. So $W_3/Z_3 \ne 0/0$ as claimed. Also the points of $P_1$ and $P_2$ are affine points, so from proposition \ref{P:TEDW001}  we have  $w(P_1+P_2)=W_3/Z_3$. 
~~~ \qed
\end{proof}
Assume that, $p=2^{255}-19$, $a=486664$ and $d=486660$. The twisted Edwards curve $\E_{\TE,a,d}$  is a complete Edwards curve over $\F_p$ of order $8\ell$, where $\ell$ is the prime 
\begin{gather*}
\ell=72370055773322622139731865630429942408\\
\qquad  57116359379907606001950938285454250989.  
\end{gather*}

The cost of the mixed differential addition and doubling formulas \eqref{eq:P6} is
$3\M+6\s+3\D$.


\begin{remark}
	Let $\E_{\TE,a,d}$ be a complete twisted Edwards curve over $\F_q$ with $\chi(d)=\chi(ad)=-1$. Then, $\E_{\TE,a,d}$ has the four torsion subgroup as 
	$$\E_{\TE,a,d}(\F_q)[4]=\{(0,1), (0,-1), (1/\sqrt{a}, 0), (-1/\sqrt{a}, 0) \}.$$ 
	Then the coset of the point $P=(x,y)$ on the curve up to this subgroup equals  
	$$P+\E_{\TE,a,d}(\F_q)[4]=\{(x,y), (-x,-y),  (y/\sqrt{a}, -x\sqrt{a}), (-y/\sqrt{a}, x\sqrt{a}) \}.$$ 
	We note that the proposed $w$-function has the property that $w(Q)=w(P)$ for all points $Q$ in the coset of $P$. This differential addition formulas could be applicable with the eliminating cofactors technique through point compression \cite{H15}
\end{remark}


\paragraph{{\bf Retrieving $x$ and $y$ from $w$-Coordinates.}}
Assume taht $P=(x_1,y_1),Q=(x_2,y_2),Q-P=(x_0,y_0), w(P)=w_1,w(Q)=w_2$ and $w(Q-P)=w_0$. 
From doubling formula \ref{eq:TE} and $w$-function $w(x,y)=dx^2y^2$,  we have 
 \begin{equation}\label{eq:4kP}
4(x_1,y_1)=(\dfrac{4x_1y_1(ax_1^2-y_1^2)(w_1^2-1)}{4w_1(w_1+1)^2-4ew_1^2+(w_1^2-1)^2},
\dfrac{2ew_1(w_1^2+1)-(w_1+1)^4}{4w_1(w_1+1)^2-4ew_1^2-(w_1^2-1)^2} )
\end{equation}
Where $e=4a/d$.

Also, if $x_0y_0(ax_0^2-y_0^2) \ne 0$ then
\begin{equation}\label{eq:Recovery}
x_1y_1(ax_1^2-y_1^2)=\dfrac{ w_2(w_1w_0-1)^2 -w_0^2-w_1^2-(4-8a/d)w_0w_1-2}{2dx_0y_0(ax_0^2-y_0^2)}
\end{equation}

So, one can use this formula to recover $4P$ given $Q-P,w(P),w(Q)$.
In particular, we can recover $4kP$ given $P,w(kP),w((k+1)P)$. So the Montgomery ladder can be used not only to compute $w(kP)$ given $w(P)$, but also to compute $4kP$ given $P$. Algorithm \ref{J:Recovery} describes the Montgomery ladder approach for point multiplication of $4kP$.

\begin{algorithm}
\caption{\; Recovering $4kP$ from $w$-coordinates dADD}\label{J:Recovery}
\begin{algorithmic}[1]
 \Statex {\bf Inputs :}
 {Elliptic curve  $\E_{\TE,a,d}/F_{q}$,~ Point $P=(x,y) $ on $\E_{\TE,a,d}$ } and integer k
 \Statex {\bf Output : $Q=4kP$} 
 \Statex {}
 \State  set: $w_0=dx^2y^2$
   \State compute: $w_1\leftarrow w(kP), w_2\leftarrow w((k+1)P) $ \Comment{dADD }
   \State Compute: $x_1y_1(ax_1^2-y_1^2)$ from Eq.\ref{eq:Recovery}  \Comment{$kP=(x_1,y_1)$ }
   \State Compute $4kP$ from Eq. \ref{eq:4kP}  
   \State \Return $Q=4kP$

\end{algorithmic}
\end{algorithm}

Let $k=(k_{m-1},\cdots, k_1,k_0)$ and $k'=(k_{m-1},\cdots, k_2)$, then  $k=4k'+2k_1+k_0$. Scaler multiplication can be modified  as in Algorithm \ref{J:Recovery2} for computing $kP$ when given $P,w(k'P)$ and $w((k'+1)P)$.

\begin{algorithm}
\caption{\; Retrieving  $kP$ from $w$-coordinates dADD}\label{J:Recovery2}
\begin{algorithmic}[1]
 \Statex {\bf Inputs :} \Statex {\hspace{32pt}  Elliptic curve  $\E_{\TE,a,d}/F_q$}\Comment{Twisted Edwards curve over $F_q$ }
\Statex {\hspace{32pt} $P=(x,y) \in \E_{\TE,a,d}$  } 
 \Statex {\hspace{32pt} $ k=(k_{m-1},\cdots, k_1,k_0)$} \Comment{$k$ is a positive integer}

 \Statex {\bf Output : $Q=kP$} 
 \Statex {}
 \State set: $k'=(k_{m-1},\cdots, k_2)$
 \State  set: $w_0=dx^2y^2$  
   \State compute: $w_1\leftarrow w(k'P), w_2\leftarrow w((k'+1)P) $ \Comment{dADD }
   \State Compute: $x_1y_1(ax_1^2-y_1^2)$ from Eq.\ref{eq:Recovery}  \Comment{$k'P=(x_1,y_1)$ }
   \State Compute $4k'P$ from doubling formula \ref{eq:4kP}  
  \State set: $Q'=4k'P $  
  \State set: {$Q=Q'+(2k_1+k_0)P$}\Comment{$k=4k'+2k_1+k_0$ }

   \State \Return $Q$

\end{algorithmic}
\end{algorithm}

\begin{proposition}\label{P:TEDW2}
	Let $w$ be the function on twisted Edwards curve  $\E_{\TE,a,d}$ over $\F_q$ given by $w(x,y)=ax^2/y^2$. Let $P_1$ , $P_2$ be two points of twisted Edwards curve  $\E_{\TE,a,d}$. Consider the $w$-coordinate dADD for $P_1$ , $P_2$. If  $w_1w_2 \ne 1 $ and $w_1^2 \ne 1$, then we have 
\begin{equation*}\label{eq:D1}
	w_4	= \dfrac{4w_1((w_1+1)^2-ew_1)}{(w_1^2-1)^2 }~,
	\end{equation*}
	
	\begin{equation*}
	\label{eq:B1}
	w_3w_0 = \dfrac{(w_1-w_2)^2}{(w_1w_2-1)^2}~,
	\end{equation*}
and
\begin{equation*}
\label{eq:C01}
w_3+w_0 = \dfrac{2(w_1+w_2)(w_1w_2+1)-4(e-2)w_1w_2}{(w_1w_2-1)^2},
\end{equation*}

	where $e=4d/a$.
\end{proposition}

\begin{proof}
From theorem \ref{thm:TED}, the twisted Edwards curve 
$\E_{\TE,a,d}$ is birationally equivalent to $\E_{\TE,d,a}$ via the map $(x,y) \rightarrow (x,1/y)$. So $w(x,y)=ax^2/y^2$ is equivalent to $w(x,y)=dx^2y^2$. \qed
\end{proof}
Similarly, we obtain the same projective and mixed $w-$coordinates formulas as \eqref{eq:P1}, \eqref{eq:P2}, \eqref{eq:P4} and \eqref{eq:P6}. 

\begin{proposition}\label{P:TEDW02}
Let $a,d$ be elements of $\F_q$ such that $ad(a-d) \ne 0$ and $\chi(a)=\chi(ad)=-1$.  
	Let $w$ be the function on twisted Edwards curve over $\F_q$ given by $w(x,y)=ax^2/y^2$. The following differential addition and doubling formula is complete.
\begin{gather*}\label{eq:D02}
	w_4	= \dfrac{4w_1((w_1+1)^2-ew_1)}{(w_1^2-1)^2 }~,~w_3+w_0 = \dfrac{2(w_1+w_2)(w_1w_2+1)-4(e-2)w_1w_2}{(w_1w_2-1)^2}~.
\end{gather*}
	where $e=4d/a$.
\end{proposition}

\begin{proof}
 Note that, If $P_1=P_2$, then $w_1=w_2$ and $w_0=0$. hence $w_3=w_4$. So differential addition formula is unified. 
We need to show that the denominators are nonzero. Suppose $w_1^2-1=0$. So $ax_1^2/y_1^2=\pm 1$ and $x_1y_1 \ne 0$. If $ax_1^2/y_1^2=1$, then $\chi(a)=1$, which is a contradiction. If $ax_1^2/y_1^2=-1$, then from curve equation we have $1+dx_1^2y_1^2=0$, so $dx_1^2y_1^2=-1$. If $\chi(-1)=1$ then we have $\chi(a)=1$,  which is a contradiction. If $\chi(-1)=-1$ then we have $\chi(d)=-1$,  which is a contradiction.  So $w_1^2-1 \ne 0$.

Now suppose that $w_1w_2-1=0$. Then $a^2x_1^2x_2^2/(y_1^2y_2^2)=1$ and $x_iy_i \ne 0$. So $ax_1x_2/y_1y_2=\pm 1$. 
Let $ax_1x_2/y_1y_2=1$,so $x_1=y_1y_2/ax_2$. Substituating into curve equation yields
$$1+dx_1^2y_1^2=\dfrac{y_1^2(1+dx_2^2y_2^2)}{ax_2^2}.$$
 Therefore,
 $$(1+\sqrt{d}x_1y_1)^2=\dfrac{y_1^2(1+dx_2^2y_2^2)}{ax_2^2}+2\sqrt{d}x_1y_1=\frac{1}{a}\left(\dfrac{y_1(1+\sqrt{d} x_2y_2)}{x_2}\right)^2.$$
 Since $\chi(a)=-1$, this must reduce to $0=0$. So $1+\sqrt{d}x_1y_1=0.$ 
 
 Similary, 
 $$(1-\sqrt{d}x_1y_1)^2=\frac{1}{a}\left(\dfrac{y_1(1-\sqrt{d} x_2y_2)}{x_2}\right)^2.$$
 Which implies that $1-\sqrt{d}x_1y_1=0.$ Therefore, $x_1=y_1=0$. which is a contradiction.
 
 The case where $ax_1x_2/y_1y_2=-1$ similary produces a contradiction. Therefore the differential addition and doubling formulas is always defined for points in $\E_{\TE,a,d}(F_q).$ \qed
 \end{proof}
\begin{proposition}\label{P:TEDW5}
	Let $w$ be the function on twisted Edwards curve  $\E_{\TE,a,d}$ over $\F_q$ given by $w(x,y)=\sqrt{ad}\; x^2$. Let $P_1$ , $P_2$ be two points of twisted Edwards curve  $\E_{\TE,a,d}$. Consider the $w$-coordinate dADD for $P_1$ , $P_2$. If  $w_1w_2 \ne 1 $ and $w_1^2 \ne 1$, then we have 
	\begin{equation}\label{eq:TA2}
	w_4	= \dfrac{4w_1((w_1+1)^2-ew_1)}{(w_1^2-1)^2 }~~~,~~~
	w_3w_0 = \dfrac{(w_1-w_2)^2}{(w_1w_2-1)^2}~.
	\end{equation}
where $e=\dfrac{(d+\sqrt{ad})^2}{d\sqrt{ad}}$.
\end{proposition}
\begin{proof}
The twisted Edwards curve $\E_{\TE,a,d}$ and Montgomery curve $\M_{A,B}$ where $A=-(a+d)/\sqrt{ad},\, B=16\sqrt{ad}/(a-d)^2$ are isogenous 
by the map\\
\begin{gather*}\label{TEMo2}
\phi(x,y)=(\dfrac{1}{\sqrt{ad}\;x^2},\dfrac{(a-d)^2y}{4ad\; x(y^2-1)}).
\end{gather*}
But the Montgomery curve $\M_{A,B}$ has dADD-M formulas with the function $w$ by $w(X,Y)=X$ and $e=2-A$. 
 So from map $\phi$ we have $w(x,y)=\dfrac{1}{\sqrt{ad}\; x^2}$ . Then, for function $\omega=1/w$ on twisted Edwards curve $TE_{a,d}$ we have $w(x,y)=\sqrt{ad}\; x^2$ and $e=\dfrac{(d+\sqrt{ad})^2}{d\sqrt{ad}}$. \qed
\end{proof}

Furthermore, for twisted Edwards curves where $\chi(ad)=1$, we define another differential formulas by the rational function $w$ by $w(x,y)=\sqrt{ad}\left(\dfrac{2xy}{1+dx^2y^2}\right)^2$.
\begin{proposition}\label{P:TEDW3}
	Let $w$ be a function on twisted Edwards curve  $\E_{\TE,a,d}$ over $\F_q$ given by $w(x,y)=\sqrt{ad}\left(\dfrac{2xy}{1+dx^2y^2}\right)^2$. Let $P_1$ , $P_2$ be two points of twisted Edwards curve  $\E_{\TE,a,d}$. Consider the $w$-coordinate dADD for $P_1$ , $P_2$. If for $i=0,1,2,3,4,$ $w_i \in \F_q$, $w_1w_2 \ne 1 $ and $w_1^2 \ne 1$, then we have 
	\begin{gather*}\label{eq:A00}
	w_4	= \dfrac{4w_1((w_1+1)^2-ew_1)}{(w_1^2-1)^2 }~~~~~,~~~~w_3w_0 = \dfrac{(w_1-w_2)^2}{(w_1w_2-1)^2}~,
	\end{gather*}
	where $e=2+(a+d)/\sqrt{ad}$.
\end{proposition}

\begin{proof}
	This function is well computed for all affine points on a twisted Edwards curve .
	Since $-(x,y)=(-x,y)$, for all points $P$ on the curve, we have
	$w(P)=w(-P)$. Also, we have $w(\cO)=0$.
	Let $P_1=(x_1,y_1)$ be a point of twisted Edwards curve  $\E_{\TE,a,d}$. From the doubling formula \eqref{eq:TE} we have
$$x_4^2y_4^2=\dfrac{4x_1^2y_1^2(y_1^2-ax_1^2)^2}{(1+dx_1^2x_2^2)^2(1-dx_1^2y_1^2)^2}=\dfrac{4x_1^2y_1^2\left(~(1+dx_1^2y_1^2)^2-4ax_1^2y_1^2~\right)}{(1+dx_1^2x_2^2)^2\left(~(1+dx_1^2y_1^2)^2-4dx_1^2y_1^2~\right)}.$$
So from definition of $w-$function we obtain
$$x_4^2y_4^2=\dfrac{w_1(\sqrt{ad}-aw_1)}{\sqrt{ad}(\sqrt{ad}-dw_1)}.$$
Hence
	\begin{equation*}
	w_4=\dfrac{4w_1(\sqrt{ad}-aw_1)(\sqrt{ad}-dw_1)}{\left(~(\sqrt{ad}-dw_1)+d/\sqrt{ad} w_1(\sqrt{ad}-aw_1)~\right)^2}.
	\end{equation*}
	Now by some straightforward calculation, we obtain
		\begin{equation*}
	w_4	= \dfrac{4w_1((w_1+1)^2-ew_1)}{(w_1^2-1)^2 }.
	\end{equation*}
	where $e=2+\dfrac{a+d}{\sqrt{ad}}$.
	
	Also, from definition of $w-$function we have
	\begin{equation*} \label{eq:PA00}
	w_3w_0=\dfrac{16adx_3^2y_3^2x_0^2y_0^2}{(1+dx_3^2y_3^2)^2(1+dx_0^2y_0^2)^2}.
	\end{equation*}
	From the addition formula \eqref{eq:TE}, we have 
	\begin{equation} \label{eq:PA01}
	x_3^2y_3^2x_0^2y_0^2=\dfrac{\left(x_1^2y_1^2(y_2^4+a^2x_2^4)-x_2^2y_2^2(y_1^4+a^2x_1^4) \right)^2} {\left(1-(dx_1^2y_1^2)(dx_2^2y_2^2) \right)^4}.
	\end{equation}
	Now by substitution $y_i^4+a^2x_i^4=(1+dx_i^2y_i^2)^2-2x_i^2y_i^2$ and $w_i=\dfrac{4\sqrt{ad}~x_i^2y_i^2}{(1+dx_i^2y_i^2)^2}$  for $i=1,2$ in equation \eqref{eq:PA01}, we obtain 
	$$	16adx_3^2y_3^2x_0^2y_0^2=(w_1-w_2)^2(1+dx_1^2y_1^2)^4(1+dx_2^2y_2^2)^4.$$
So
	\begin{equation*}
	w_3w_0=(w_1-w_2)^2\left(\dfrac{(1+dx_1^2y_1^2)^2(1+dx_2^2y_2^2)^2}{(1-d^2x_1^2x_2^2y_1^2y_2^2)^2(1+dx_3^2y_3^2)(1+dx_0^2y_0^2)}\right)^2.
	\end{equation*}
On the other hands
$$(1-d^2x_1^2x_2^2y_1^2y_2^2)^2(1+dx_3^2y_3^2)(1+dx_0^2y_0^2)=(1+dx_1^2y_1^2)^2(1+dx_2^2y_2^2)^2-16adx_1^2x_2^2y_1^2y_2^2.$$

	Now by some straightforward calculation, we obtain 
$$\dfrac{(1+dx_1^2y_1^2)^2(1+dx_2^2y_2^2)^2}{(1+dx_1^2y_1^2)^2(1+dx_2^2y_2^2)^2-16adx_1^2x_2^2y_1^2y_2^2}=\dfrac{1}{1-w_1w_2}.$$
So
	\begin{equation*}
	w_3w_0=\dfrac{(w_1-w_2)^2}{(w_1w_2-1)^2} .
	\end{equation*} \qed
	%
	%
\end{proof}
 Another proof
 \begin{proof}
We know that doubling formula for twisted Edwards curve  $\E_{\TE,a,d}$ is an isogeny from $\E_{\TE,a,d}$ to $\E_{\TE,a,d}$. \\

$\phi(x,y)=2(x,y)=\left(\dfrac{2x_1y_1}  {1+dx_1^2y_1^2},\,\dfrac{y_1^2-ax_1^2}  {1-dx_1^2y_1^2}  \right).$\\

But from proposition \ref{P:TEDW5} the twisted Edwards curve  $\E_{\TE,a,d}$ over $\F_q$ have $w$-function $w(x,y)=\sqrt{ad}\; x^2$ and $e=\dfrac{(d+\sqrt{ad})^2}{d\sqrt{ad}}$. So from isogeny map $\phi$, we have $w(x,y)=\sqrt{ad}\left(\dfrac{2xy}{1+dx^2y^2}\right)^2$ and $e=\dfrac{(d+\sqrt{ad})^2}{d\sqrt{ad}}$.
\end{proof}
\vspace*{3pt}
So, we have the same results for this $w$-coordinates by formulas \eqref{eq:P1}, \eqref{eq:P2}, \eqref{eq:P4} and \eqref{eq:P6}. 
\begin{remark}
	Let $\E_{\TE,a,d}$ be a  twisted Edwards curve over $\F_q$ with $\chi(ad)=1$. Then, $\E_{\TE,a,d}$ has the full 2-torsion subgroup 
and  the coset of the point $P=(x,y)$ on the curve up to this subgroup equals   
	$$P+\E_{\TE,a,d}(\F_q)[2]=\{(x,y), (-x,-y), (\frac{1}{\sqrt{ad}\;x}, \frac{a}{\sqrt{ad}\;y}), (\frac{-1}{\sqrt{ad}\;x}, \frac{-a}{\sqrt{ad}\;y}) \}.$$ 

Note that, the $w$-function $w(x,y)=\sqrt{ad}\left(\dfrac{2xy}{1+dx^2y^2}\right)^2$ is invariant for the coset of a point up to the full 2-torsion subgroup of the incomplete twisted Edwards curve $\E_{\TE,a,d}$ over $\F_q$ with $\chi(ad)=1$. 
\end{remark}
\begin{proposition}\label{P:TEDW7}
	Let $w$ be the function on twisted Edwards curve  $\E_{\TE,a,d}$ over $\F_q$ given by $w(x,y)= y$. Let $P_1$ , $P_2$ be two points of twisted Edwards curve  $\E_{\TE,a,d}$. Consider the $w$-coordinate dADD for $P_1$ , $P_2$. Then we have 
	\begin{equation}\label{eq:TA07}
	w_4	= \dfrac{w_1^4-(t+1)(w_1^2-1)^2}{1+t(w_1^2-1)^2}~~~,~~~
	w_3w_0 = \dfrac{ w_1^2w_2^2-(t+1)(w_1^2-1)(w_2^2-1)}{1+t(w_1^2-1)(w_2^2-1)}~,
	\end{equation}
Where $t=d/(a-d)$.
\end{proposition}
\begin{proof}
This $w$-function is well computed for all affine points on a twisted Edwards curve .
	Also, we have $w(P)=w(-P)$ and  $w(\cO)=1$.
	Let $P_1=(x_1,y_1)$  and $P_2=(x_2,y_2)$ be two points of twisted Edwards curve  $\E_{\TE,a,d}$. Then from the addition and doubling formula \eqref{eq:TE} we have
\begin{equation*}
\begin{array}{c}
	w_4=y_4=\dfrac{ax_1^2-y_1^2}{1-dx_1^2y_1^2} ,\vspace*{7pt}\\

	w_3w_0=y_3y_0=\dfrac{y_1^2y_2^2-a^2x_1^2x_2^2 } {1-d^2x_1^2x_2^2y_1^2y_2^2}.
\end{array}
	\end{equation*}
	Also from curve equation, we can write $x_i^2=(1-y_i^2)/(a-dy_i^2)~(i=1,2)$.Note that $a-dy_i^2 \ne 0$. Since from curve equation, if $a-dy_i^2=0$ then we have $a=d$, which is a contradiction.
	So by substitution $w_i=y_i^2$ and $x_i^2=(1-w_i^2)/(a-dw_i^2)$, we have
	\begin{equation*}
	w_4	= \dfrac{-dw_1^4+2aw_1^2-a}{ dw_1^4-2dw_1^2+a}=\dfrac{ w_1^4-a/(a-d)~(w_1^2-1)^2}{1+d/(a-d)~(w_1^2-1)^2 }.
	\end{equation*}
	So, by seting $t=d/(a-d)$ we obtain
	\begin{equation*}
	w_4	=\dfrac{w_1^4-(t+1)~(w_1^2-1)^2}{1+t~(w_1^2-1)^2}.
	\end{equation*}
And 
\begin{equation*}
	w_3w_0	= \dfrac{ -dw_1^2w_2^2+aw_1^2+aw_2^2+a }{ dw_1^2w_2^2-dw_1^2-dw_2^2+a}= \dfrac{w_1^2w_2^2-(1+t)~(w_1^2-1)(w_2^2-1)}{ 1+t(w_1^2-1)(w_2^2-1)}.
	\end{equation*}
 \qed
\end{proof}
\begin{proposition}\label{P:TEDW8}
	Let $w$ be the function on twisted Edwards curve  $\E_{\TE,a,d}$ over $\F_q$ given by $w(x,y)=\sqrt[4]{d/a}\; y$. Let $P_1$ , $P_2$ be two points of twisted Edwards curve  $\E_{\TE,a,d}$. Consider the $w$-coordinate dADD for $P_1$ , $P_2$. Then we have 
	\begin{equation}\label{eq:TA08}
	w_4	= \dfrac{2w_1^2-r(w_1^4+1)}{r((w_1^4+1)-2rw_1^2)}~~~,~~~
	w_3w_0 = \dfrac{2w_1w_2- r(w_1^2w_2^2+1)}{(w_1^2w_2^2+1)-2rw_1w_2}~,
	\end{equation}
Where $r=\sqrt[4] {d/a}$.
\end{proposition}
\begin{proof}
The proof is similar to the proposition \ref{P:TEDW7}. \qed
\end{proof}
As before, assume that $w_0$ is given as a field element, and the inputs $w_1, w_2$ are given as fractions
$W_1/Z_1$ , $W_2/Z_2$ and the outputs $w_4, w_3 $ are given as fraction $W_4/Z_4$ and  $W_3/Z_3$. 

From differential and doubling formulas~\eqref{eq:TA07} the explicit projective formulas are given by
\begin{equation*}
\begin{array}{c}
 \dfrac{W_4}{Z_4}= \dfrac{W_1^4-(t+1)(W_1^2-Z_1^2)^2} {Z_1^4+t(W_1^2-Z_1^2)^2} , \vspace*{7pt}\\
 \dfrac{W_3}{Z_3}= \dfrac{Z_0~\left( W_1^2W_2^2-(t+1)(W_1^2-Z_1^2)(W_2^2-Z_2^2) \right)} {W_0 \left( Z_1^2Z_2^2+t(W_1^2-Z_1^2)(W_2^2-Z_2^2) \right)}\enspace. 
\end{array}
\end{equation*}
The following mixed projective
$w$-coordinates differential addition and doubling formulas have the total cost of $4\M+7\s+2\D $ :
\begin{equation}
\label{eq:CoTE08}
\begin{array}{c}
A_1=(W_1^2-Z_1^2),\ B_1=W_1^4,\ C_1=Z_1^4\vspace*{3pt}\\
 A_2=(W_2^2-Z_2^2),\ B_2=W_1^2W_2^2,\ C_2=Z_1^2Z_2^2  \vspace*{3pt}\\
W_4= B_1-A_1^2-tA_1^2, \ Z_4= C_1+tA_1^2,~~\vspace*{7pt}\\
W_3=B_2-A_1A_2-tA_1A_2~~,~~Z_3=w_0(C_2+tA_1A_2) \enspace. 
\end{array}
\end{equation}
Also from differential and doubling formulas~\eqref{eq:TA08} and by seting $s=(1+r^2)/(1-r^2)$, the explicit projective formulas are given by

\begin{equation*}\label{eq:PCoTE}
\begin{array}{c}
 \dfrac{W_4}{Z_4}= \dfrac{(W_1^2+Z_1^2)^2-s(W_1^2-Z_1^2)^2} {r\left(  (W_1^2+Z_1^2)^2+s(W_1+Z_1)^2\right)} , \vspace*{7pt}\\
 \dfrac{W_3}{Z_3}= \dfrac{Z_0~\left((W_1^2+Z_1^2)(W_2^2+Z_2^2)-s(W_1^2-Z_1^2)(W_2^2-Z_2^2)\right)} {W_0\left((W_1^2+Z_1^2)(W_2^2+Z_2^2)-s(W_1^2-Z_1^2)(W_2^2-Z_2^2)\right)}\enspace. 
\end{array}
\end{equation*}
The following mixed projective
$w$-coordinates differential addition and doubling formulas have the total cost of $3\M+6\s+3\D $ :
\begin{equation}
\label{eq:CoTE008}
\begin{array}{c}
A_1=(W_1^2-Z_1^2),\ B_1=(W_1^2+Z_1^2),\vspace*{3pt}\\
A_2=(W_2^2-Z_2^2) ,\ B_2=(W_2^2+Z_2^2),\vspace*{3pt}\\
W_4= B_1^2-sA_1^2 , \ Z_4= r(B_1^2+sA_1^2),~~\vspace*{7pt}\\
W_3=B_1B_2-sA_1A_2~~,~~Z_3=w_0(B_1B_2+sA_1A_2)\enspace. 
\end{array}
\end{equation}
In the next proposition, we show that if $\chi(a)=\chi(d)=-1$, then differential addition formula \ref{eq:CoTE008} is complete.
\begin{proposition}\label{P:TEDW07}
Fix a field  $\F_q$ with odd charactristic and $a,d \in \F_q $ such that $ad(a-d) \ne 0$ and $\chi(a)=\chi(d)=-1$.
Define $w:\E_{\TE,a,d}(\F_q) \mapsto \F_q \cup  \{\infty \}$ as fallow : $w(x,y)=\sqrt[4]{d/a}\; y$. 

Let $W_i,Z_i~(i=0,...,4)$  be elements of $\F_q$ and $r^4=d/a , s=(1+r^2)/(1-r^2)$. Define
\begin{equation*}
\label{eq:PJ07}
\begin{array}{c}
A_1=(W_1^2-Z_1^2),\ B_1=(W_1^2+Z_1^2),\vspace*{3pt}\\
A_2=(W_2^2-Z_2^2) ,\ B_2=(W_2^2+Z_2^2),\vspace*{3pt}\\
W_4= B_1^2-sA_1^2 , \ Z_4= r(B_1^2+sA_1^2),~~\vspace*{7pt}\\
W_3=B_1B_2-sA_1A_2~~,~~Z_3=w_0(B_1B_2+sA_1A_2)\enspace. 
\end{array}
\end{equation*}
a) Let $P_1$ be an element of $\E_{\TE,a,d}$ and  $w(P_1)=W_1/Z_1$, then $W_4/Z_4 \ne 0/0$ and $w(2P_1)=W_4/Z_4$. \\
b) Let $P_1$ and $P_2$ be two element of $\E_{\TE,a,d}$ and  $w(P_1)=W_1/Z_1 ,w (P_2)=W_2/Z_2$,  then $W_3/Z_3 \ne 0/0$ and $w(P_1+P_2)=W_3/Z_3$. 

Here $W/Z$ means the quotient of $W$ and $Z$ in $\F_q$  if $Z \ne 0$; it means $\infty$ if $W \ne 0$ and $Z=0$; it is undefined if $W=Z=0$.
\end{proposition}
\begin{proof}

Since $\chi(a)=\chi(d)=-1$, so for every affine point $P=(x,y)$ we have $y \ne 0$ and so $w(P) \ne 0$. Also the only points at infinity are $\infty_2=(1:0:\pm \sqrt{a/d}:0)$, where $w(\infty_2)=(\pm 1/r:1)$. So for every point of $\F_q$-rational projective points of $\E_{\TE,a,d}$ , we have $WZ \ne 0$.

a) If $W_4=Z_4=0$, then $A_1=B_1=0$ and so $W_1=Z_1=0$, contradiction. So $W_4/Z_4 \ne 0/0$ as claimed.
 On the other hands, if $P_1=\infty_2$, then $2P_1=(0,1)$ so $w(2P_1)=W_4/Z_4=(r:1)$. If $P_1$ be an affine point, then from proposition \ref{P:TEDW7}  $w(2P_1)=W_4/Z_4$ as claimed.
 
 b) If  $W_3=Z_3=0$, then $A_1A_2=B_1B_2=0$. If $A_1=B_1=0$ or $A_2=B_2=0$ then $W_1=Z_1=0$ or $W_2=Z_2=0$,contradiction.  If $A_1=B_2=0$ then $W_1^2=Z_1^2$ and $W_2^2=-Z_2^2$.
 So $w(P_1)^2=1$ and $w(P_2)^2=-1$. So $r^2y_1^2=1$ and $r^2y_2^2=-1$. Therefore we have $\chi(r^2)=\chi(-r^2)=1$. Now by substituation $y_1^2=1/r^2$ in curve equation we have $ax_1^2=-1/r^2$, which implies that $\chi(a)=1$, contradiction. The case where $A_2=B_1=0$ similary produces a contradiction. So $W_3/Z_3 \ne 0/0$ as claimed.
~~~ \qed
\end{proof}
We knowe that, a twisted Edwards curve $\E_{\TE,a,d}$ is birationally equivalent to the Montgomery curve $\E_{\M,A,B}$ by the map $\psi:\E_{\TE,a,d} \to \E_{\M, A,B}$
\begin{equation*}
\psi(x,y)=\Big(\frac{1+y}{1-y},\frac{1+y}{x(1-y)} \Big).
\end{equation*}
where $A=2(a+d)/(a-d),\, B=4/(a-d)$. 
Also, the Montgomery curve $\E_{\M,A,B}$ is birationally equivalent to the twisted Edwards
curve $\E_{\TE,a,d}$ by the inverse map 
\begin{equation*}
\psi^{-1}(X,Y)=\Big(\frac{X}{Y},\frac{X-1}{X+1} \Big),
\end{equation*}
where $a=(A + 2)/B,\, d = (A - 2)/B$.

Now the $y$-coordinate of $\psi^{-1}(X,Y)$ only depends on $X$, and the $X$-coordinate of $\psi(x,y)$ only depends on $y$. In projective coordinates this correspondence becomes remarkably simple:
$$\psi:(y,z)\longmapsto (z+y:z-y)~,~\psi^{-1}:(X,Z)\longmapsto (X-Z:X+Z).$$

Therefore switching between $X$-coordinate of Montgomery curve and $y$-coordinate of Twisted Edwards curve is quasi cost free. So one can use this and obtain the following differential addition and doubling formulas for twisted Edwards curve $\E_{\TE,a,d}$ with $w$-function $w(x,y)=y$: 

\begin{equation}\label{eq:MTE1}
\begin{array}{c}
A_1=W_1^2,\ B_1=Z_1^2, \ A_2=(W_1Z_2+W_2Z_1)^2,\ B_2=(W_1Z_2+W_2Z_1)^2, \vspace*{7pt}\\
W_4=2A_1B_1-(B_1-A_1) \left(~A_1+B_1-(a+d)/(a-d)~(B_1-A_1) \right), \vspace*{7pt}\\
Z_4=2A_1B_1+(B_1-A_1) \left(~A_1+B_1-(a+d)/(a-d)~(B_1-A_1) \right), \vspace*{7pt}\\
 W_3=Z_0(A_2-B_2)-W_0(A_2+B_2), \vspace*{7pt}\\
Z _3=Z_0(A_2+B_2)-W_0(A_2-B_2) \enspace. 
\end{array}
\end{equation}
From \eqref{eq:MTE1}, the total cost of the mixed differential addition and doubling is 
$5\M+4\s+1\D$. 


Furthermore, for the twisted Edwards curves $\E_{\TE,a,d}$ with $\chi(a/d)=1$,  
the cost of the following mixed differential addition and doubling formulas is  $3\M+6\s+3\D$. Here we let $r^2=a/d$.  
\begin{equation}
\label{eq:MTE3}
\begin{array}{c}
A_1=W_1^2,\ B_1=Z_1^2, \ H_1=(rB_1+A_1),\ H_1=(rB_1-A_1), \vspace*{7pt}\\
G=H_1+H_2, \ K=H_1-H_2, \ S=1/r~K, \ T=rK,  \vspace*{7pt}\\
W_4=T-G, \ Z_4=G-S \vspace*{7pt}\\
 W_3=Z_0(A_2-B_2)-W_0(A_2+B_2), \vspace*{7pt}\\
Z _3=Z_0(A_2+B_2)-W_0(A_2-B_2) \enspace. 
\end{array}
\end{equation}
\begin{proposition}\label{P:TEDW9}
	Let $w$ be the function on twisted Edwards curve  $\E_{\TE,a,d}$ over $\F_q$ given by $w(x,y)= y^2$. Let $P_1$ , $P_2$ be two points of twisted Edwards curve  $\E_{\TE,a,d}$. Consider the $w$-coordinate dADD for $P_1$ , $P_2$. Then we have 
	\begin{equation}\label{eq:TA09}
	w_4	= \dfrac{\Big(w_1^2-(t+1)(w_1-1)^2\Big)^2}{\Big(1+t(w_1-1)^2\Big)^2}~~~,~~~
	w_3w_0 = \dfrac{\Big( w_1w_2-(t+1)(w_1-1)(w_2-1)\Big)^2}{\Big(1+t(w_1-1)(w_2-1)\Big)^2}~,
	\end{equation}
Where $t=d/(a-d)$.
\end{proposition}
\begin{proof}
This $w$-function is well computed for all affine points on a twisted Edwards curve .
	Also, we have $w(P)=w(-P)$ and  $w(\cO)=1$.
	Let $P_1=(x_1,y_1)$  and $P_2=(x_2,y_2)$ be two points of twisted Edwards curve  $\E_{\TE,a,d}$. Then from the addition and doubling formula \eqref{eq:TE} we have
		\begin{equation*}
\begin{array}{c}
	w_4=y_4^2=\Big (\dfrac{ax_1^2-y_1^2}{1-dx_1^2y_1^2} \Big)^2, ~~\\

	w_3w_0=(y_3y_0)^2=\dfrac{\Big(y_1^2y_2^2-a^2x_1^2x_2^2 \Big)^2} {\Big(1-d^2x_1^2x_2^2y_1^2y_2^2 \Big)^2}.
\end{array}
	\end{equation*}
	Also from curve equation, 
	We can write $x_i^2=(1-y_i^2)/(a-dy_i^2)~(i=1,2)$.
	So by substitution $w_i=y_i^2$ and $x_i^2=(1-w_i)/(a-dw_i)$, we have
	\begin{equation*}
	w_4	= \dfrac{(dw_1^2-2aw_1+a)^2}{( dw_1^2-2dw_1+a)^2}=\dfrac{\Big(  w_1^2-a/(a-d)~(w_1-1)^2 \Big)^2}{\Big( 1+d/(a-d)~(w_1-1)^2 \Big)^2}.
	\end{equation*}
	So, by seting $t=d/(a-d)$ we obtain
	\begin{equation*}
	w_4	=\dfrac{\Big ( w_1^2-(t+1)~(w_1-1)^2 \Big)^2}{\Big( 1+t~(w_1-1)^2 \Big)^2}.
	\end{equation*}

And 
\begin{equation*}
	w_3w_0	= \dfrac{\Big( dw_1w_2-aw_1-aw_2+a \Big)^2}{\Big( dw_1w_2-dw_1-dw_2+a \Big)^2}=\dfrac{\Big(  w_1w_2-(1+t)~(w_1-1)(w_2-1) \Big)^2}{\Big( 1+t(w_1-1)(w_2-1) \Big)^2}.
	\end{equation*}
 \qed
\end{proof}

\begin{proposition}\label{P:TEDW10}
	Let $w$ be the function on twisted Edwards curve  $\E_{\TE,a,d}$ over $\F_q$ given by $w(x,y)=\sqrt{d/a}\; y^2$. Let $P_1$ , $P_2$ be two points of twisted Edwards curve  $\E_{\TE,a,d}$. Consider the $w$-coordinate dADD for $P_1$ , $P_2$. Then we have 
	\begin{equation}\label{eq:TA10}
	w_4	= \dfrac{\Big(r(w_1^2+1)-2w_1\Big)^2}{r\Big((w_1^2+1)-2rw_1\Big)^2}~~~,~~~
	w_3w_0 = \dfrac{\Big( r(w_1w_2+1)-(w_1+w_2)\Big)^2}{\Big((w_1w_2+1)-r(w_1+w_2)\Big)^2}~,
	\end{equation}
Where $r=\sqrt{d/a}$.
\end{proposition}
\begin{proof}
The proof is similar to the proposition \ref{P:TEDW7}. \qed
\end{proof}
As before, assume that $w_0$ is given as a field element, and the inputs $w_1, w_2$ are given as fractions
$W_1/Z_1$ , $W_2/Z_2$ and the outputs $w_4, w_3 $ are given as fraction $W_4/Z_4$ and  $W_3/Z_3$. 

From Eq.~\eqref{eq:TA09} the explicit projective formulas are given by
\begin{equation*}
\begin{array}{c}
 \dfrac{W_4}{Z_4}= \dfrac{\Big(W_1^2-(t+1)(W_1-Z_1)^2\Big)^2} {\Big(Z_1^2+t(W_1-Z_1)^2\Big)^2} , \vspace*{7pt}\\
 \dfrac{W_3}{Z_3}= \dfrac{Z_0~\Big(W_1W_2-(t+1)(W_1-Z_1)(W_2-Z_2)\Big)^2} {W_0\Big(Z_1Z_2+t(W_1-Z_1)(W_2-Z_2)\Big)^2}\enspace. 
\end{array}
\end{equation*}
The following mixed projective
$w$-coordinates differential addition and doubling formulas have the total cost of $4\M+7\s+2\D $ :
\begin{equation}
\label{eq:CoTE}
\begin{array}{c}
A_1=(W_1-Z_1),\ B_1=W_1^2,\ C_1=Z_1^2\vspace*{3pt}\\
 A_2=(W_2-Z_2),\ B_2=W_1W_2,\ C_2=Z_1Z_2  \vspace*{3pt}\\
W_4= (B_1-A_1^2-tA_1^2)^2, \ Z_4= (C_1+tA_1^2)^2,~~\vspace*{7pt}\\
W_3=(B_2-A_1A_2-tA_1A_2)^2~~,~~Z_3=w_0(C_2+tA_1A_2)^2 \enspace. 
\end{array}
\end{equation}
Also from differential and doubling formulas~\eqref{P:TEDW10} and by seting $s=(1+r)/(1-r)$, the explicit projective formulas are given by

\begin{equation*}
\begin{array}{c}
 \dfrac{W_4}{Z_4}= \dfrac{\Big(s(W_1-Z_1)^2-(W_1+Z_1)^2\Big)^2} {r\Big(s(W_1-Z_1)^2+(W_1+Z_1)^2\Big)^2} , \vspace*{7pt}\\
 \dfrac{W_3}{Z_3}= \dfrac{Z_0~\Big(s(W_1-Z_1)(W_2-Z_2)-(W_1+Z_1)(W_2+Z_2)\Big)^2} {W_0\Big(s(W_1-Z_1)(W_2-Z_2)+(W_1+Z_1)(W_2+Z_2)\Big)^2}\enspace. 
\end{array}
\end{equation*}
The following mixed projective
$w$-coordinates differential addition and doubling formulas have the total cost of $3\M+6\s+3\D $ :
\begin{equation}
\label{eq:CoTE09}
\begin{array}{c}
A_1=(W_1-Z_1),\ B_1=(W_1+Z_1),\vspace*{3pt}\\
A_2=(W_2-Z_2) ,\ B_2=(W_2+Z_2),\vspace*{3pt}\\
W_4= (~sA_1^2 -B_1^2)^2, \ Z_4= r(~sA_1^2 + B_1^2)^2,~~\vspace*{7pt}\\
W_3=(sA_1A_2-B_1B_2)^2~~,~~Z_3=w_0(sA_1A_2+B_1B_2)^2 \enspace. 
\end{array}
\end{equation}
In the next proposition, we show that if $\chi(a)=\chi(d)=-1$, then differential addition formula \ref{eq:CoTE09} is complete.
\begin{proposition}\label{P:TEDW09}
Fix a field  $\F_q$ with odd charactristic and $a,d \in \F_q $ such that $ad(a-d) \ne 0$ and $\chi(a)=\chi(d)=-1$.
Define $w:\E_{\TE,a,d}(\F_q) \mapsto \F_q \cup  \{\infty \}$ as fallow : $w(x,y)=\sqrt{d/a}\; y^2$. 

Let $W_i,Z_i~(i=0,...,4)$  be elements of $\F_q$ and $r=\sqrt{d/a} , s=(1+r)/(1-r)$. Define
\begin{equation*}
\begin{array}{c}
A_1=(W_1-Z_1),\ B_1=(W_1+Z_1),\vspace*{3pt}\\
A_2=(W_2-Z_2) ,\ B_2=(W_2+Z_2),\vspace*{3pt}\\
W_4= (~sA_1^2 -B_1^2)^2, \ Z_4= r(~sA_1^2 + B_1^2)^2,~~\vspace*{7pt}\\
W_3=Z_0(sA_1A_2-B_1B_2)^2~~,~~Z_3=W_0(sA_1A_2+B_1B_2)^2 \enspace.
\end{array}
\end{equation*}
a) Let $P_1$ be an element of $\E_{\TE,a,d}$ and  $w(P_1)=W_1/Z_1$, then $W_4/Z_4 \ne 0/0$ and $w(2P_1)=W_4/Z_4$. \\
b) Let $P_1$ and $P_2$ be two element of $\E_{\TE,a,d}$ and  $w(P_1)=W_1/Z_1 ,w (P_2)=W_2/Z_2$,  then $W_3/Z_3 \ne 0/0$ and $w(P_1+P_2)=W_3/Z_3$. 

Here $W/Z$ means the quotient of $W$ and $Z$ in $\F_q$  if $Z \ne 0$; it means $\infty$ if $W \ne 0$ and $Z=0$; it is undefined if $W=Z=0$.
\end{proposition}
\begin{proof}

Since $\chi(a)=\chi(d)=-1$, so for every affine point $P=(x,y)$ we have $y \ne 0$ and so $w(P) \ne 0$. Also the only points at infinity are $\infty_2=(1:0:\pm \sqrt{a/d}:0)$, where $w(\infty_2)=(1/r:1)$. So for every point of $\F_q$-rational projective points of $\E_{\TE,a,d}$ , we have $WZ \ne 0$.

a) If $W_4=Z_4=0$, then $A_1=B_1=0$ and so $W_1=Z_1=0$, contradiction. So $W_4/Z_4 \ne 0/0$ as claimed.
 On the other hands, if $P_1=\infty_2$, then $2P_1=(0,1)$ so $w(2P_1)=W_4/Z_4=(r:1)$. If $P_1$ be an affine point, then from proposition \ref{P:TEDW7}  $w(2P_1)=W_4/Z_4$ as claimed.
 
 b) If  $W_3=Z_3=0$, then $A_1A_2=B_1B_2=0$. If $A_1=B_1=0$ or $A_2=B_2=0$ then $W_1=Z_1=0$ or $W_2=Z_2=0$,contradiction.  If $A_1=B_2=0$ then $W_1=Z_1$ and $W_2=-Z_2$.
 So $w(P_1)=1$ and $w(P_2)=-1$. So $ry_1^2=1$ and $ry_2^2=-1$. Therefore we have $\chi(r)=\chi(-r)=1$. Now by substituation $y_1^2=1/r$ in curve equation we have $ax_1^2=-1/r$, which implies that $\chi(a)=1$, contradiction. The case where $A_2=B_1=0$ similary produces a contradiction. So $W_3/Z_3 \ne 0/0$ as claimed.
~~~ \qed
\end{proof}
\subsection{Montgomery curves}
Now, we consider the Montgomery curves. Note that above $w$-coordinates differential addition and doubling formulas for twisted Edwards curves can be applied for Montgomery curve using the birational maps between these two curves \eqref{MonTE}. Furthermore, from formulas \eqref{eq:P4} and \eqref{eq:P6}, we give the mixed $x$-coordinates differential addition and doubling formulas for Montgomery curves with cost of $3\M+7\s+1\D$ and $3\M+6\s+3\D$. 

We recall \cite{Mo}, that for the Montgomery curve $\E_{\M,A,B}$ with the rational function $w(x,y)=x$, we have the following differential addition
formulas. 
\begin{gather*}\label{eq:A3}
w_4	= \dfrac{(w_1^2-1)^2 }{4w_1((w_1+1)^2-ew_1)}~~~~~,~~~~w_3w_0 = \dfrac{(w_1w_2-1)^2}{(w_1-w_2)^2}~,
\end{gather*}
where $e=2-A$. In other words, the $x$-coordinates formulas for Montgomery curves and above $w$ coordinates formulas \eqref{eq:A1} for twisted Edwards curves are inverse of each other. It means the projective formulas for Montgomery curves is obtained by the projective formulas \eqref{eq:P1} only by swapping the role of $W$ and $Z$.   
Therefore, from formulas \eqref{eq:P4} we have the following formulas with cost of $3\M+7\s+1\D$    
\begin{equation}
\label{eq:M4}
\begin{array}{c}
 A_1=(W_1+Z_1) ,\ B_1=(W_1-Z_1),\ A_2=(W_2+Z_2),\ B_2=(W_2-Z_2), \vspace*{3pt}\\
 C=A_1B_2~ ,~\ D=A_2B_1~,~\ 	E=A_1^2-B_1^2 ~,~F=(A_1^4+B_1^4)-E^2, \vspace*{3pt}\\
W_4	= F ~~~,~~~Z_4=2(A_1^4-(e/4)E^2)-F, \vspace*{3pt}  \\
W_3=w_0(C+D)^2~~~,~~~Z_3=(C-D)^2 \enspace.  
\end{array}
\end{equation}
Furthermore, for the Montgomery curves $\E_{\M,A,B}$ with $\chi(A^2-4)=1$,  
the cost of the following mixed differential addition and doubling formulas is  $3\M+6\s+3\D$.
\begin{equation}
\label{eq:M6}
\begin{array}{c}
A_1=(W_1 + Z_1)~,~ B_1=(W_1 - Z_1)~,~A_2=(W_2 + Z_2)~, B_2=(W_2 -Z_2), \vspace*{3pt}\\
C= A_1~B_2~,~D= A_2~B_1~,~H_1=(rA_1^2+B_1^2)^2~,~ H_2=(rA_1^2-B_1^2)^2, \vspace*{3pt}\\
G=(H_1+H_2)\ ,\ K=(H_1-H_2)\ ,\ S=\frac{1}{r}K\ ,\ T=rK, \vspace*{3pt}\\
W_4=T-S~~,~~Z_4=2G-S-T, \vspace*{3pt}\\
W_3=w_0(C+D)^2~~,~~Z_3=(C-D)^2 \enspace.
\end{array}
\end{equation}
Where $r=(A+2)/\sqrt{A^2-4}$. So, if the parameter $r$ is chosen to be small then the cost of mixed differential addition and doubling formulas is $3\M+6\s+1\D$.

\begin{example}
Assume that, $p=2^{255}-19$, $A=486662$ and $B=1$. The Montgomery curve $\E_{\M,A,B}$ over $\F_p$ is of order $8\ell$, where $\ell$ is the prime 
\begin{gather*}
\ell=72370055773322622139731865630429942408\\
\qquad  57116359379907606001950938285454250989.  
\end{gather*}

The cost of the mixed differential addition and doubling formulas \eqref{eq:M4} is
$3\M+7\s+1\D$,  where $1\D$ is the multiplication by the small constant $-e/4=121665$ .
 
\end{example}
\begin{proposition}\label{P:CoMont}
	Let $w$ be the function on Montgomery curve $\E_{\M,A,B}$  over $\F_q$ given by $w(x,y)=\frac{\sqrt{A^2-4}}{A+2}\; \dfrac{(x-1)^2}{(x+1)^2}$. Let $P_1$ , $P_2$ be two points of Montgomery curve $\E_{\M,A,B}$ . Consider the $w$-coordinate dADD for $P_1$ , $P_2$. Then we have 
	\begin{equation}\label{eq:CoMon2}
	w_4	= \dfrac{\Big(r(w_1^2+1)-2w_1\Big)^2}{r\Big((w_1^2+1)-2rw_1\Big)^2}~~~,~~~
	w_3w_0 = \dfrac{\Big( r(w_1w_2+1)-(w_1+w_2)\Big)^2}{\Big((w_1w_2+1)-r(w_1+w_2)\Big)^2}~,
	\end{equation}
Where $r=\frac{\sqrt{A^2-4}}{A+2}$.
\end{proposition}
if $\chi(A^2-4)=1, \chi(B(A+2))=-1$ then this formula is complete.
\begin{proof}
 \qed
\end{proof}

\begin{proposition}\label{P:CoMont2}
Fix a field  $\F_q$ with odd charactristic and $A,B \in \F_q $ such that $B(A^2-4) \ne 0$ and $\chi(A^2-4)=1, \chi(B(A+2))=-1$.
Define $w:\E_{\M,A,B}(\F_q) \mapsto \F_q \cup  \{\infty \}$ as fallow : $w(x,y)=\frac{\sqrt{A^2-4}}{A+2}\; \dfrac{(x-1)^2}{(x+1)^2}$. 

Let $W_i,Z_i~(i=0,...,4)$  be elements of $\F_q$ and $r=\frac{\sqrt{A^2-4}}{A+2} , s=(1+r)/(1-r)$. Define
\begin{equation*}
\label{eq:CoMont}
\begin{array}{c}
A_1=(W_1-Z_1),\ B_1=(W_1+Z_1),\vspace*{3pt}\\
A_2=(W_2-Z_2) ,\ B_2=(W_2+Z_2),\vspace*{3pt}\\
W_4= (~sA_1^2 -B_1^2)^2, \ Z_4= r(~sA_1^2 + B_1^2)^2,~~\vspace*{7pt}\\
W_3=(sA_1A_2-B_1B_2)^2~~,~~Z_3=w_0(sA_1A_2+B_1B_2)^2 \enspace.
\end{array}
\end{equation*}
a) Let $P_1$ be an element of $\E_{\M,A,B}$ and  $w(P_1)=W_1/Z_1$, then $W_4/Z_4 \ne 0/0$ and $w(2P_1)=W_4/Z_4$. \\
b) Let $P_1$ and $P_2$ be two element of $\E_{\M,A,B}$  and  $w(P_1)=W_1/Z_1 ,w (P_2)=W_2/Z_2$,  then $W_3/Z_3 \ne 0/0$ and $w(P_1+P_2)=W_3/Z_3$. 

Here $W/Z$ means the quotient of $W$ and $Z$ in $\F_q$  if $Z \ne 0$; it means $\infty$ if $W \ne 0$ and $Z=0$; it is undefined if $W=Z=0$.
\end{proposition}

\begin{proof}
a) If $W_4=Z_4=0$, then $A_1=B_1=0$ and so $W_1=Z_1=0$, contradiction. So $W_4/Z_4 \ne 0/0$ as claimed.
 
 b) If  $W_3=Z_3=0$, then $A_1A_2=B_1B_2=0$. If $A_1=B_1=0$ or $A_2=B_2=0$ then $W_1=Z_1=0$ or $W_2=Z_2=0$,contradiction.  If $A_1=B_2=0$ then $W_1=Z_1$ and $W_2=-Z_2$.
 So $w(P_1)=1$ and $w(P_2)=-1$. So $r(\frac{x_1-1}{x_1+1})^2=1$ and $r(\frac{x_2-1}{x_2+1})^2=-1$. Therefore we have $\chi(r)=\chi(-r)=1$. Now by substituation $y_1^2=1/r$ in curve equation we have $ax_1^2=-1/r$, which implies that $\chi(a)=1$, contradiction. The case where $A_2=B_1=0$ similary produces a contradiction. So $W_3/Z_3 \ne 0/0$ as claimed.
~~~ \qed
 
\end{proof}


\section{Concluding Remarks}
The known Montgomery ladder differential addition formulas for elliptic curves over a finite field are not complete; they work for all input points $P$ except for the case where $w(P)$ equals $(1:0)$ or $(0:1)$. However, the Montgomery ladder algorithm works perfectly in cryptographic applications, since the order of base point $P$ should be a large prime number. The cost of the Montgomery-like formulas 
is $5\M+4\s+1\D$ if the base point $P$ is affine. We believe, this record can be obtained
for any form of elliptic curve with group order divisible by 4 by a suitable rational function. 
This includes the family of Jacobi curves.  

Our proposed Montgomery-like formulas for
twisted Edwards curves are improved in terms of efficiency and speed. 
They are almost complete formulas if the curve parameters are chosen carefully. 
The mixed formulas are provided for twisted Edwards curves with the cost of $3\M+7\s+1\D$. Also, faster mixed formulas are presented for a subfamily of twisted Edwards curves with the cost of $3\M+6\s+3\D$ which gives further speedup if the parameters are chosen to be small.   

In Table 1, we compare our new differential addition formulas with the known formulas for other forms of elliptic
curves. Notice, the fast and efficient presented formulas by Gaudry-Lubicz \cite{GL9} and Bernstein-Lange \cite{BLEFD} are given with the cost of $4\M+6\s+3\D$, and $3\M+6\s+3\D$ if the base point is affine, only for subfamily of elliptic curves with 3 points of order 2. Our formulas have the same costs and presented for a subfamily of twisted Edwards with a point of order 4 which includes the complete twisted Edwards curves therein. 

For complete twisted Edwards curves, the proposed $w$ functions are invariant in the coset of a point $P$ with respect to the subgroup of $\F_q$-rational points with order 4. And, for incomplete twisted Edwards curves the suggested $w$ function is invariant in the coset of a point $P$ up to the subgroup of full 2-torsion points. 
For future works, we are going to investigate the use of these differential addition formulas along with the eliminating cofactors technique through point compression \cite{H15}. Computing the full point representation at the end of Montgomery ladder is an alternative question which is useful for cryptographic applications that need the full version of the scalar multiplication algorithm.    
\begin{table}[H]
	\caption{Cost of differential addition and doubling for families
		of elliptic curves in odd characteristic}\label{tab:difadd}
	\centering
	\begin{tabular}{|l|l|l|l|}
		\hline
		 Model~~ & Projective differential~~ & Mixed differential~~&Completeness \\
		\hline\hline
		Montgomery~W=X~ \cite{Mo}~~ &
		$6\M+4\s+1\D$~W=X~&$5\M+4\s+1\D$~~&No\\
	Montgomery~W=X~ \cite{BLEFD}~~ &
		$4\M+3\s+3\D$~~&$3\M+6\s+3\D$~~&No~~\\
		this work~W=X~\eqref{eq:M4}&$4\M+7\s+1\D$~~&$3\M+7\s+1\D$&No~~\\
		this work~W=X~\eqref{eq:M6}&$4\M+6\s+3\D$~~&$3\M+6\s+3\D$&No~~\\
		\hline
		Kummer curve~~\cite{GL9}~~&$4\M+6\s+3\D$~~&$3\M+6\s+3\D$&No~~ \\
		\hline
		Edwards curves &&&\\
		$(c=1,d=r^2)$~W=Y~\cite{BLEFD} &$4\M+6\s+5\D$ ~&$3\M+6\s+5\D$~&No~~\\
		$(c=1,d=r^2)$~W=Y~\cite{BLEFD}&$4\M+6\s+3\D$~~&$3\M+6\s+3\D$&No~~\\
		Jacobi quartic  &&&\\
	$(\epsilon=1,\delta=r^2) ~W=X^2~$\cite{GGX12}&$6\M+4\s+1\D$~&$5\M+4\s+1\D$~&No~~\\
		\hline
		Twisted Edwards &&&\\
	     this work~$W=dX^2Y^2 $\eqref{eq:PC2}&$6\M+6\s+2\D$~~&$5\M+6\s+2\D$&Yes~~\\
		this work~\eqref{eq:P2}&$6\M+4\s+1\D$~~&$5\M+4\s+1\D$~&No~~\\
		this work~\eqref{eq:P4}&$4\M+7\s+1\D$~~&$3\M+7\s+1\D$&No~~\\
	    this work~\eqref{eq:P6}&$4\M+6\s+3\D$~~&$3\M+6\s+3\D$&No~~\\
		\hline
	\end{tabular}
\end{table}

\vspace{5mm} \noindent \textbf{Acknowledgment.} The authors would like to thank anonymous reviewers
for their useful comments. This research was in part supported by a grant from IPM (No. 95050416).

\end{document}